\documentclass[11pt,final]{amsart} 
\usepackage{etex}
\usepackage[utf8]{inputenc} 
\usepackage{color}
\usepackage{comment}
\usepackage{tikz-cd}
\usepackage{mathtools}
\usepackage{calligra}

\usepackage{graphicx} 

\usepackage{array} 
\usepackage{paralist} 
\usepackage[numbers]{natbib}
\usepackage{hyperref}

\usepackage{amsmath}
\usepackage{amssymb}
\usepackage{amsthm}
\usepackage{mathrsfs}
\usepackage{enumerate}
\usepackage[all]{xy}
\usepackage{todonotes}
\usetikzlibrary{calc}

\theoremstyle{definition}
\numberwithin{equation}{subsection}
\newtheorem{thm}{Theorem}[section]
\newtheorem{prop}[thm]{Proposition}
\newtheorem{cor}[thm]{Corollary}

\newtheorem{lem}[thm]{Lemma}

\newtheorem{defn}[thm]{Definition}
\newtheorem{ex}[thm]{Example}

\newtheorem{remk}[thm]{Remark}

\newcommand{\tn}[1]{\textnormal{#1}}

\newcommand{\wt}[1]{\widetilde{#1}}
\newcommand{\wh}[1]{\widehat{#1}}
\newcommand{\vocab}[1]{\textbf{#1}}

\newcommand{\bbz}{\mathbb Z}

\newcommand{\bbn}{\mathbb N}

\newcommand{\bbc}{\mathbb C}

\newcommand{\vspan}{\tn{span}}

\newcommand{\diff}{\backslash}

\newcommand{\coker}{\tn{coker}}

\newcommand{\Hom}{\tn{Hom}}
\newcommand*{\shom}{\mathcal{H}\kern -.5pt om}

\newcommand{\bbl}{\mathbb{L}}

\newcommand{\spec}{\textnormal{Spec}}

\newcommand{\Ad}{\textnormal{Ad}}

\newcommand{\Gr}{\mathrm{Gr}}

\newcommand{\ord}{\text{ord}}
\newcommand{\dord}{\text{dord}}
\newcommand{\cK}{{\mathcal{K}}}
\newcommand{\cA}{{\mathcal{A}}}

\begin{document}
\title[Commutative algebras of Fractional Differential Operators]
{A Burchnall--Chaundy--Krichever Theory for \\ Fractional Differential Operators}

\author[casper]{W. Riley Casper}
\address{Department of Mathematics, California State University Fullerton, CA 92831, USA}
\email{\href{wcasper@fullerton.edu}{wcasper@fullerton.edu}}
\author[horozov]{Emil Horozov}
\address{
Department of Mathematics and Informatics, Sofia University,
5. J. Bourchier Blvd., Sofia 1126, and, 
Institute of Mathematics and Informatics, 
Bulg. Acad. of Sci., Acad. G. Bonchev Str., Block 8, Sofia 1113, Bulgaria}
\email{\href{horozov@fmi.uni-sofia.bg}{horozov@fmi.uni-sofia.bg}}
\author[iliev]{Plamen Iliev}
\address{School of Mathematics, Georgia Institute of Technology, Atlanta, GA 30332, USA}
\email{\href{iliev@math.gatech.edu}{iliev@math.gatech.edu}}
\author[yakimov]{Milen Yakimov}
\address{Department of Mathematics, Northeastern University, Boston,
MA 02115 \\
USA}
\email{\href{m.yakimov@northeastern.edu}{m.yakimov@northeastern.edu}}
\thanks{The research of W.R.C. was supported by an RSCA intramural grant 0359121 from CSUF. E.H. was supported by Bulgarian Science Fund grant DN02/05. P.I. was supported by Simons Foundation grant \#635462.  M.Y. was supported by NSF grants DMS-1901830 
and DMS-2131243.}

\begin{abstract}
Fractional differential (and difference) operators play a role in a number of diverse settings, including integrable systems, 
mirror symmetry, Hurwitz numbers, the Bethe ansatz equations. 
We prove extensions of the three major results on algebras of commuting (ordinary) differentials operators
to the setting of fractional differential operators: (1) the Burchnall--Chaundy theorem that a pair of 
commuting differential operators is algebraically dependent, 
(2) the Krichever correspondence classifying maximal commutative algebras of differential operators 
by an algebro--geometric construction (in the rank 1 case) and 
(3) the characterization of such algebras in terms of Sato's Grassmannian.
Unlike the available proofs of the Burchnall--Chaundy theorem which use the action of one differential operator 
on the kernel of the other, our extension to the fractional case uses bounds on orders of fractional differential 
operators and growth of algebras. This also presents a new and much shorter proof of the original result. 
The second main theorem is achieved by developing a new tool of the spectral field of a point in Sato's Grassmannian, 
which carries more information than the widely used notion of spectral curve of a KP solution.
Our Krichever type correspondence for commutative algebras of fractional differential operators of rank 1 is based on infinite jet bundles.
\end{abstract}

\maketitle
\section{Introduction}
\subsection{Commutative algebras of differential operators}
Algebras of commuting ordinary differential operators play a fundamental role in many areas of mathematics 
and mathematical physics,
ranging from integrable systems \cite{Segal-Wilson,vanMoerbeke} to algebraic geometry, where the underlying algebro-geometric 
structures were used in the solution of the Schottky problem \cite{Shiota}. 
From now on, for brevity, by a differential operator we will mean an ordinary differential operator. 

There are three fundamental results on commutative algebras of differential operators:
\begin{enumerate}
\item[(Thm1)] The {\em{Burchnall--Chaundy theorem}} \cite{BC} (from almost 100 years ago) establishes that every two 
commuting differential operators, at least one of which is of positive order, are algebraically dependent.
\item[(Thm2)] The {\em{Krichever correspondence}} gives an explicit construction of all maximal commutative algebras of differential 
operators in terms of algebro-geometric data: a projective curve $X$, a maximally torsion free invertible sheaf $\mathcal L$ on $X$, 
a smooth point $\infty$ on $X$, a parametrization of a neighborhood of $\infty$ and the invertible sheaf on it, and some additional 
data in the higher rank case. 
The construction is due to Krichever in the rank 1 smooth case \cite{Krichever1}. It was extended to the singular rank 1 case 
by Mumford \cite{Mumford} and to higher rank by Krichever \cite{Krichever2}.
The rank of an algebra of commuting differential operators is defined to be the dimension of the space of common eigenfunctions 
for generic eigenvalues. It is well known to be equal to the 
greatest common divisor of the orders of the operators in it.
\item[(Thm3)] {\em{Sato's theory}} \cite{Sato} parametrizes the solutions of the Kadomtsev--Petviashvili (KP) hierarchy in terms of the 
points of an infinite dimensional Grassmannian $\Gr$, called Sato's Grassmannian. To every plane $W \in \Gr$, 
one associates a spectral curve and a commutative algebra of ordinary differential operators $\cA_W$, which is 
isomorphic (as an algebra) to the coordinate ring $A_W$ of the spectral curve. A formal version of Krichever's construction of Baker--Akhiezer functions
\cite{Krichever1} gives the following characterization of maximal commutative algebras of differential operators:
up to a change of variable and a conjugation by a function, all maximal commutative algebras of differential 
operators are of the form $\cA_W$ for a plane $W \in \Gr$ with non-trivial spectral algebra $A_W$, i.e. $A_W \neq \mathbb C$. 
\end{enumerate}

A {\em{fractional differential operator}} is an operator that can be represented in the form $P Q^{-1}$ for two differential operators $P$ and $Q$. 
In more conceptual terms a fractional differential operator is an element of the {\em{skew field of fractions
of the algebra of differential operators}} (which is well known to be an Ore domain \cite{Krichever3}), see \cite[Ch. 6]{GW}
for background on Ore domains and skew field of fractions. 

In the last 25 years there has been a great interest in fractional differential and difference operators from diverse points of view, 
a partial list of which is as follows:
\begin{enumerate} 
\item[(a)] They form the phase space of Krichever's \cite{Krichever3} rational reductions of the KP hierarchy
(also called the constrained KP hierarchies), which contain as very special cases all Gelfand--Dickey reductions. 
These integrable hierarchies have been in turn much studied, see e.g. \cite{Dickey,HvdL}. 
\item[(b)] In \cite{BH} it was proved that the generating functions of weighted Hurwitz numbers are given in terms of constraint KP $\tau$-functions whose 
Baker--Akhiezer functions are the Meijer $G$-functions (which are eigenfunctions of fractions of two hypergeometric operators). 
\item[(c)] A reproduction procedure for constructing new solutions of the Bethe ansatz equation for the Gaudin model for the Lie superalgebra $\mathfrak{gl}(m|n)$ 
out of old ones was discovered in \cite{HMVY}, where it was shown that in the fermionic picture the reproduction is governed by fractional differential operators 
(which in \cite{HMVY} are called rational pseudodifferential operators). 
\item[(d)] Rational reductions of the 2D Toda hierarchy were defined and studied in \cite{BCRR} and related to Frobenius structures and mirror symmetry; 
the full descendent all-genus Gromov–Witten potentials of certain crepant resolutions were conjectured to be 
given by $\tau$-functions of the reduction.
\end{enumerate}
However, not much is known about the structure of algebras of commuting fractional differential operators, which are certainly in 
the background of these results. The known proofs of (Thm1)--(Thm3) 
do not generalize to the case of fractional differential operators and no such generalizations are currently known.

In this paper we obtain extensions of (Thm1)--(Thm3) to commutative algebras of fractional differential operators, with the exception of the higher rank case 
of (Thm2). The theorems are proved in full generality without any additional assumptions.

{\bf{Setting:}} If a fractional differential operator of nonzero order has analytic coefficients in a neighborhood of some point in $\bbc$ and invertible leading coefficient, then by a change of variable we can assume that its leading coefficient is $1$, and the remaining coefficients are analytic in a neighborhood of $0$. Furthermore, conjugating by a function, we can make the sub-leading coefficient $0$. Thus, without any restriction, we can assume that the fractional differential operator is in {{\em{normalized form}}, which means that the leading coefficient is $1$, the sub-leading coefficient is $0$, and the coefficients are analytic in a neighborhood of $0$. We can relax the analyticity condition, and from now we will work with fractional differential operators with coefficients in $\bbc[[x]]$.

\subsection{A Burchnall--Chaundy theorem for fractional differential operators}
The known proofs of the Burchnall--Chaundy theorem use the action of one of the commuting 
differential operators on the kernel of the other. This approach is not applicable to fractional differential operators 
as there is no analog of kernel that can be used in this fashion. 

We first give a new proof of the Burchnall--Chaundy theorem that relies on an upper bound of the order of any polynomial 
in two commuting differential operators and uses it to bound the growth of the algebra formed by them.
It is much shorter than the known proofs of the theorem. We then obtain upper and lower bounds 
on the order of any polynomial in two commuting fractional differential operators. The lower 
bound requires a fundamentally new idea that relies on the denominatorial order 
of a fractional differential operator which is defined to be the minimal order 
of a right denominator for it. These bounds are used for an algebra growth estimate 
which in turn leads to the following:
\medskip

\noindent
{\bf{Theorem A.}}
{\em{Let $ P(x,\partial_x)$ and $ Q(x,\partial_x)$ be commuting fractional differential operators, where $P$ is monic of nonzero order and $Q$ is not a constant.
Then $P(x,\partial_x)$ and $ Q(x,\partial_x)$ are algebraically dependent.}}
\medskip

Theorem A excludes only the case of pairs $(P,Q)$ such that both $P$ and $Q$ have order $0$ with algebraically independent 
leading coefficients. Obviously, if both operators have order 0, they may not be functionally dependent, e.g. take $P=1+x$ and $Q=e^x$.

\subsection{A characterization in terms of Sato's theory} 
Next we obtain a characterization of maximal algebras of commuting 
fractional differential operators in terms of Sato's theory. 
(We work with the formal Sato Grassmannian, see \cite{vanMoerbeke} for background, and not with the 
analytic Segal-Wilson Grassmannian \cite{Segal-Wilson}.)
The pivotal 
ingredient is a new notion of {\em{spectral field}} of a plane $W \in \Gr$ 
which is of independent interest for the study of Sato's Grassmannian 
and the dynamics of the KP flows.
Recall that the {\em{spectral algebra}} of a plane $W \in \Gr$ is defined by 
\[
A_W = \{f(z)\in \bbc((z^{-1})) : f(z)W\subseteq W\}.
\]
Define the {\em{spectral field}} of a plane $W \in \Gr$ by 
\[
K_W = \{f(z)\in \bbc((z^{-1})): \dim (W + f(z)W)/W < \infty\}.
\]
Also define the {\em{rank}} of $W\in\Gr$ to be the dimension of the $K_W$-subspace of $\bbc((z^{-1}))$ spanned by $W$.
We investigate in detail the structure of $K_W$, showing among other things the following:
\begin{enumerate}
\item[(a)] $K_W$ is a field;
\item[(b)] If $W \in \Gr$ is a plane with nontrivial spectral algebra $A_W \neq \bbc$, 
then $K_W$ is the fraction field of $A_W$.
\item[(c)] The rank of $W$ is finite if and only if its spectral field $K_W$ is nontrivial, i.e. $K_W \neq \bbc$.
\end{enumerate}
We also give an example of a plane $W \in \Gr$ for which $A_W = \bbc$ and $K_W$ has transcendence 
degree 1 over $\bbc$. To each $W \in \Gr$, we associate a field of fractional differential operators ${\mathscr{K}}_W$
which is isomorphic to $K_W$ as follows.  
If the Baker--Akhiezer function of $W$ is given by
\[
\psi_W(x,z) = \Big(1 + \sum_{j=1}^\infty u_j(x)z^{-j}\Big)e^{xz}
\]
(see Sect. \ref{2.1} for details), then we set
\begin{equation}
\label{scrKW}
{\mathscr{K}}_W=\{U(x,\partial_x)f(\partial_x)U(x,\partial_x)^{-1}: f(z)\in K_W\},
\end{equation}
where $U(x,\partial_x)= 1 + \sum_{j=1}^\infty u_j(x)\partial_x^{-j}$. As common in Sato's theory, the expression for $\psi_W(x,z)$ 
is viewed formally as a $\mathcal{D}$-module and not as a product of convergent series in $z^{-1}$ and $z$.
\medskip

\noindent
{\bf{Theorem B.}}
{\em{Let $W$ be a plane in the big cell $\Gr_+(0)$ of Sato's Grassmannian. If $K_W \neq \bbc$, then ${\mathscr{K}}_W$
is a maximal commutative algebra of fractional differential operators.

All maximal commutative algebras of fractional differential operators containing a normalized fractional differential operator $L(x, \partial_x)$ of nonzero order
arise in this way.
}}
\medskip

While (Thm3) is easily derived from Krichever's construction and results on Baker--Akhiezer functions \cite{Krichever1}, this is not the case with 
Theorem B. It requires a series of new results on the Sato's Grassmannian and the spectral fields of its planes. In turn, Theorem B
plays a key role in our extension of (Thm2) to fractional differential operators. 

\subsection{A Krichever type correspondence for commutative algebras of fractional differential operators of rank one}
We start with the following data:
\begin{enumerate}
\item[(i)] An algebraic curve $X$; 
\item[(ii)] A smooth point $p\in X$; 
\item[(iii)] A line bundle $\mathcal L$ over $X$ with trivial cohomology;
\item[(iv)] A local coordinate $z^{-1}$ of $X$ in an analytic neighborhood $U$ of $p$; 
\item[(v)] A local trivialization $\varphi: \mathcal L\rightarrow \mathcal O_X(-1)$ over $U$.
\end{enumerate}
However, this data from the classical machinery is not sufficient 
to construct in an algebro-geometric fashion all commutative algebras of fractional differential operators. 
We add a key additional ingredient to it associated to semi-infinite jet bundles, dealt with in the algebro-geometric category.
Denote by $\mathcal J^m(\mathcal L)$
the bundle of $m$-jets of $\mathcal L$ and consider the directed system of sheaves defined by the push-forward maps
$\pi_{m,\ell}: \mathcal J^m(\mathcal L)\rightarrow\mathcal J^\ell(\mathcal L)$, $m \geq \ell \geq 0$. 
The {\em{infinite jet bundle}} of $\mathcal L$ is the sheaf on $X$ given by projective limit
\[
\mathcal J^\infty(\mathcal L) = \varprojlim_m \mathcal J^m(\mathcal L).
\]
The maps $\pi_{m,\ell}$ have canonical splittings $\iota_{m,\ell}: \mathcal J^\ell(\mathcal E)\rightarrow\mathcal J^m(\mathcal E)$ and we can consider the colimit
\[
\mathcal J^{\infty,0}(\mathcal L) = \varinjlim_m \mathcal J^m(\mathcal L).
\]
which we call the {\em{semi-infinite jet bundle}} of $\mathcal L$. There is a canonical map $j^\infty:  \Gamma(U,\mathcal L) \to \Gamma(U,\mathcal J^\infty(\mathcal L))$, 
called the {\em{$\infty$-jet of the section $s$}} over an open subset $U$ of $X$, defined in \eqref{j_infty}.

The paper can be read without prior knowledge of jet bundles. For the convenience of the reader, 
in the appendix we describe all constructions on finite and infinite jet bundles that are needed for the paper, 
based on minimal algebro-geometric background from parts of \cite{Hartshorne}.

Let $p$ be a smooth point of the curve $X$ as in condition (ii).
We define a {\em{rational section}} of the dual of the semi-infinite jet bundle $\mathcal J^{\infty,0}(\mathcal L)$ over an open subset $U\subseteq X$ 
to be a section $\chi$ of the dual of $\mathcal J^{\infty,0}(\mathcal L)$ over a formal annulus $Z_p$ of $p$, 
which extends to a section of the dual of $\mathcal F$ over $X\diff\{p\}$ for a submodule $\mathcal F$ of $\mathcal J^{\infty,0}(\mathcal L)$ of finite codimension,
see Definition \ref{fract-sect} and Sect. \ref{5.2} for details. The following theorem describes our generalized Krichever correspondence 
for commutative algebras of fractional differential operators of rank 1.
\medskip

\noindent
{\bf{Theorem C.}}
{\em{
Consider a hextuples $(X, p,\mathcal L, z^{-1},\varphi,\chi)$ where the first 5 components of the datum satisfy (i)-(v) above (classical part of the datum) and the last one is
\begin{enumerate}
\item[(vi)] $\chi$, a rational section of the dual of the semi-infinite jet bundle over $X\diff\{p\}$
\end{enumerate}
(non-classical part of the datum). Then
\[
W = \{\varphi(\chi(j^\infty(s))): s\in \Gamma(Z_p,\mathcal L)\}
\]
defines a rank $1$ element of the big cell $\Gr_+(0)$ of Sato's Grassmannian. Furthermore every rank $1$ element of $\Gr_+(0)$ arises in this way.
}} 
\medskip

We note that by Theorem B, the plane $W\in \Gr_+(0)$ constructed in Theorem C gives rise to the maximal algebra $\mathscr K_W$ of commuting fractional differential operators, 
defined in \eqref{scrKW}. This produces an algebro-geometric construction of all rank 1 maximal algebras of commuting differential operators.

The algebras of commuting fractional differential operators of rank $>1$ are vastly more complicated than those in the differential case. 
In the latter case Krichever obtained an algebro-geometric classification in \cite[Theorem 2.3]{Krichever2}, but
a possible extension of Theorem C to the higher rank case appears to be very difficult
and one needs substantial additional data in the classification of those algebras.

Finally, we obtain the following important corollary from the proof of Theorem C:
\medskip

\noindent
{\bf{Corollary D.}}
{\em{
If $\mathscr A$ is a rank $1$ commutative algebra of fractional differential operators, then there exists a commutative algebra
of differential operators $\mathscr A_0$ and  a fractional differential operator $D(x,\partial_x)$ such that
$$D(x,\partial_x)\mathscr A D(x,\partial_x)^{-1} \subseteq \{L_1(x,\partial_x)^{-1}L_2(x,\partial_x): L_j(x,\partial_x)\in\mathscr A_0,\ j=1,2\}.$$
}}
It is easy to show that the statement of the corollary does not hold for higher rank commutative algebras of 
fractional differential operators.

In a forthcoming publication we will use the constructions in Theorem C to describe explicitly in algebro-geometric terms the evolution of the flows of the constraint 
KP hierarchy of Krichever \cite{Krichever3}. This is not a straightforward application of Theorem C and requires a number of additional arguments.

Fractional differential operators also give rise to an extension of the Duistermaat--Gr\"unbaum bispectral problem \cite{DG} where 
both spectral equations are replaced with generalized eigenvalue problems. The results in this paper can be used to 
classify all rank one solutions of the problem, leading to a fractional version of the classical Wilson's adelic Grassmannian defined in \cite{W1,W2}.
(We note that the classifying variety is different from the union of the quiver varieties in \cite{BGK} generalizing the Calogero--Moser strata of the  
Wilson's adelic Grassmannian.) This will be described in a forthcoming publication.  

In \cite{BW} it was proved that the orbits of the action of the automorphism group of the 
first Weyl algebra on the set of its one-sided ideals can be identified with the Calogero--Moser strata of Wilson's adelic Grassmannian. 
It is not clear to us precisely which algebra should be used in place of the first Weyl algebra to extend this orbit-correspondence to the case 
of fractional differential operators.

\medskip
\noindent
{\bf Acknowledgement.} We are grateful to Igor Krichever whose detailed comments and suggestions were of great help to us in improving the presentation
of the paper.
\section{Algebras of commuting differential operators and Sato's Grassmannian}
\subsection{The Burchnall--Chaundy theorem and Sato's Grassmannian}
\label{2.1}
The systematic study of commuting algebras of differential operators goes back to the seminal work of Burchnall and Chaundy \cite{BC}.  
Given a pair of commuting differential operators $L_1(x,\partial_x)$ and $L_2(x,\partial_x)$, Burchnall and Chaundy consider the \emph{simultaneous} action of $L_1$ and $L_2$ 
on the eigenspaces of $L_1$.  Their analysis constructs a nonconstant polynomial $F(z,w)$ with $F(L_1,L_2)$ acting trivially on each eigenspace and thus is identically zero, leading to the following theorem:
\begin{thm}[Burchnall--Chaundy \cite{BC}]\label{bctheorem}
Every pair of commuting differential operators is algebraically dependent.
\end{thm}
We give a second proof of this theorem in Sect. \ref{bcsect} based on a completely different approach.  

The Burchnall--Chaundy theorem shows that pairs of commuting differential operators $(L_1(x,\partial_x), L_2(x,\partial_x))$
have an associated irreducible algebraic curve $X=\{(z,w)\in\bbc^2: F(z,w)=0\}$, called the Burchnall--Chaundy curve.
It consists of all pairs $(z,w) \in \bbc$ such that $f(z,w)=0$ for all $f \in \bbc[z,w]$ satisfying $f(L_1,L_2)=0$.
We also refer to this curve as to the \vocab{spectral curve} of the commuting operators, since it is parametrized by the joint eigendata of the commuting operators.
More precisely, we may construct a family of joint eigenfunctions 
$\psi(x,z)$ satisfying $L_i(x,\partial_x)\cdot\psi(x,z) = \lambda_i(z)\psi(x,z)$ for some functions $\lambda_1(z)$ and $\lambda_2(z)$.
The algebraic relation between $L_1(x,\partial_x)$ and $L_2(x,\partial_x)$
implies that $F(\lambda_1(z),\lambda_2(z))=0$, so $z\mapsto (\lambda_1(z),\lambda_2(z))$ parametrizes $X$.

The modern theory of commuting differential operators includes a complete classification of the algebras of commuting differential operators in terms of an infinite-dimensional space 
called \vocab{Sato's Grassmannian} $\Gr$.
Informally, algebras of commuting differential operators are enlarged to algebras of commuting formal pseudodifferential operators and this data 
is recorded by the asymptotic expansions at infinity of their common eigenfunctions, which uniquely determine the initial commutative algebras.  
The points of $\Gr$ are realized as certain subspaces of the space of formal Laurent series in $z^{-1}$
\[
\mathbb L = \bbc((z^{-1})).
\]
It has a natural decomposition as 
\[
\mathbb L = \mathbb L_+\oplus \mathbb L_-,  \quad \mbox{where} \quad \mathbb L_+ = \bbc[z], \mathbb L_- = z^{-1}\mathbb[[z^{-1}]].
\]
Let $\pi_+: \mathbb L\rightarrow\mathbb L_+$ be the associated projection map.
Sato's Grassmannian is
$$\Gr = \{W\subseteq \mathbb L: \dim(\ker\pi_+|_W) <\infty,\ \ \dim(\coker\pi_+|_W)<\infty\}.$$
Mostly, we will be concerned with the big cell of index $0$, denoted $\Gr_+(0)$, consisting of $W\in\Gr$ wherein $\pi_+|_W$ is an isomorphism.

The interpretation of $\Gr$ as a classifying space for commutative algebras of differential operators uses formal pseudo-differential operators with coefficients in $\bbc[[x]]$, i.e.
formal sums of the form
\[
U(x,\partial_x) = \sum_{j=-\infty}^n u_j(x)\partial_x^j, \quad u_j(x) \in \bbc[[x]].
\]
The set $\mathcal P$ of pseudodifferential operators with coefficients in $\bbc[[x]]$ forms an algebra with product rule
\[
\partial_x^k u(x) = \sum_{j=0}^\infty \binom{k}{j} u^{(j)}(x) \partial_x^{k-j}, 
\]
which contains the algebra $\mathcal D$ of differential operators with coefficients in $\bbc[[x]]$ as a subalgebra. 
The invertible operators in $\mathcal P$ are precisely the ones whose leading terms are invertible elements of $\bbc[[x]]$. As first observed by Schur \cite{Schur}, any normalized pseudodifferential operator $L(x,\partial_x)$ of order $m\neq 0$ can be conjugated by a monic pseudodifferential operator $U(x,\partial_x)$ of order $0$ 
into $\partial_x^m$:
\begin{equation}
\label{Schur}
L(x,\partial_x) = U(x,\partial_x) \partial_x^m U(x,\partial_x)^{-1}, \; \; \mbox{where} \; \; U(x,\partial_x) = 1 + \sum_{j=1}^{\infty} u_j(x)\partial_x^{-j}.
\end{equation}
This implies that the centralizers of differential operators are necessarily commutative algebras.

We can endow $\mathbb L$ with a right $\mathcal P$-module structure by identifying $\mathbb L$ with $\mathcal P/x\mathcal P$, where $z^j$ represents the equivalence class of $\partial_x^j$.
This action satisfies 
\[
z^j\cdot \partial_x^m = z^{j+m} \; \; \mbox{and} \; \; z^j\cdot x^k = k!\binom{j}{k}z^{j-k}, \quad \forall j, m \in \bbz, k \in \bbn.
\]
Here and below 
\[
\bbn= \{0, 1, \ldots \}.
\]
Under this action, any $W\in \Gr_+(0)$ has an expression of the form $W = \mathbb L_+\cdot U(x,\partial_x)$ for some monic pseudodifferential operator of order $0$, $U(x,\partial_x)= 1 + \sum_{j=1}^{\infty} u_j(x)\partial_x^{-j}$.
The corresponding family of eigenfunctions is 
\[
\psi_W(x,z) = \Big(1 + \sum_{j=1}^\infty u_j(x)z^{-j}\Big)e^{xz}, 
\]
and is called the \vocab{(stationary) Baker--Akhiezer function} of $W$.

The pivotal algebraic construction allowing us to relate $W$ to a differential operator is a certain commutative algebra $A_W$, called the \vocab{spectral algebra} of $W$ and defined by
$$A_W = \{f(z)\in\mathbb L: f(z)W\subseteq W\}.$$
It gives rise to the commutative algebra of differential operators
\[
\cA_W = \{U(x,\partial_x)f(\partial_x)U(x,\partial_x)^{-1}: f(z)\in A_W\}.
\]
The Baker--Akhiezer function $\psi_W(x,z)$ is a family of joint eigenfunctions for this algebra.
In general $A_W$ is defined for any point in $\Gr$, but typically $A_W=\bbc$.
Later, we will introduce a novel extension of this fundamental concept, associating a certain commutative field $K_W$ with each point $W$, which we call the spectral field of $K_W$.
The extension $K_W$ is the fraction field of $A_W$ when $\bbc\subsetneq A_W$, but notably $K_W$ may be larger than $\bbc$ when $A_W=\bbc$.

The algebra of differential operators defined by the previous paragraph is maximal in the sense that it is the centralizer of a differential operator.
More generally we can consider Schur pairs $(W,A)$ consisting of a point $W\in\Gr_+(0)$ and a subalgebra $A$ of $A_W$.
Any commutative algebra of differential operators will be a subalgebra of the centralizer of one of its elements; thus Schur pairs $(W,A)$ classify all algebras of differential operators
in the sense that each such algebra is of the form 
\[
\{U(x,\partial_x)f(\partial_x)U(x,\partial_x)^{-1}: f(z)\in A\}.
\]
for a Schur pair $(W,A)$.

\subsection{The KP hierarchy}
\label{2.2}
Sato's Grassmannian is endowed with an infinite family of compatible flows called \vocab{KP (Kadomtsev-Petviashvili) flows}, which form a deep connection between 
algebras of commuting differential operators and integrable systems.
The KP flows are defined in terms of the infinite collection of commuting vector fields $\{X_n\}_{n=1}^\infty\subseteq T\Gr_+(0)$ defined by 
\[
X_{n,W}: W\mapsto \mathbb L/W,\ \ v(z)\mapsto z^nv(z).
\]
Here, the tangent space at a point $W\in\Gr_+(0)$ is given analogously to the finite-dimensional Grassmannians by $T_W\Gr_+(0) = \Hom_{\bbc}(W,\mathbb L/W)$.  

The KP flow $W(\vec t )$ corresponding to the $n$-th vector field satisfies the infinite system of partial differential equations 
$\frac{\partial}{\partial t_n}W(\vec t) = z^nW(\vec t)$ for $\vec t= (t_1,t_2,\dots)$.
If we define the (unique) pseudodifferential operators $U(x,\partial_x;\vec t)$ by $W(\vec t) = \mathbb L_+\cdot U(x,\partial_x;\vec t)$ and set
\[
L(x,\partial_x;\vec t) := U(x,\partial_x;\vec t)\partial_x U(x,\partial_x;\vec t)^{-1} = \partial_x + \sum_{j=1}^\infty a_j(x;\vec t)\partial_x^{-j}
\]
(following Schur's result \eqref{Schur}), 
then the coefficients $a_j(x;\vec t)$ satisfy an infinite system of nonlinear partial differential equations called the KP hierarchy.
The KP hierarchy may also be conveniently expressed in terms of a commutation relation called the \vocab{Lax formulation}
\begin{equation}\label{lax equation}
\frac{\partial}{\partial t_n} L(x,\partial_x;\vec t) = [(L(x,\partial_x;\vec t)^n)_+,L(x,\partial_x;\vec t)],
\end{equation}
where here $Q(x,\partial_x)_+$ denotes the differential component of a pseudodifferential operator $Q(x,\partial_x)$.

\subsection{The Krichever correspondence}
\label{2.3}
As described in \S \ref{2.1}, commuting differential operators define algebraic curves.
Krichever's construction \cite{Krichever1} is an algebro-geometric construction that allows us to go in the opposite direction and associate a commutative algebra of differential operators to an algebraic curve.
In this way, we can establish a correspondence between certain geometric data on algebraic curves and commutative algebras of differential operators.
While this correspondence works in general, we will first focus on the case of nonsingular curves.

To begin, we consider tuples of data of the from $(X,{\infty},z,D)$, where
\begin{enumerate}
\item[(i)]    $X$ is a compact Riemann surface of genus $g$;
\item[(ii)]   ${\infty}\in X$ is  a fixed point, and $z^{-1}$ is a local parameter near ${\infty}$;
\item[(iii)]  $D=P_1+\cdots +P_g$ is a non-special divisor on $X\setminus {\infty}$.
\end{enumerate}
To determine a commutative algebra of differential operators corresponding to $(X,\infty,z,D)$, it suffices to construct a Schur pair $(W,A)$.
To obtain $W$, we associate a Baker--Akhiezer function to this quadruple.

The stationary Baker--Akhiezer function corresponding to $(X,{\infty},z,D)$ is the unique
function $\psi(x,\cdot )$ on $X$, having the following two properties:
\begin{enumerate}
\item[(a)]  It is a meromorphic function on $X\setminus \{\infty\}$ with poles 
at $P_1,\dotsc,P_g$;
\item[(b)] Near ${\infty}$ it has the form 
\begin{equation*}
\psi(x,z)=\Bigg(1+\sum_{j=1}^{\infty}\frac{u_j(x)}{z^j}\Bigg)e^{xz}.
\end{equation*}
\end{enumerate}
The associated pseudodifferential operator $U(x,\partial_x) = 1 + \sum_{j=1}^\infty u_j(x)\partial_x^{-j}$ defines a point $W=\mathbb L_+\cdot U(x,\partial_x)$ of $\Gr_+(0)$.
Moreover, the algebra $A$ of holomorphic functions on $X\setminus\{\infty\}$ with a pole at $\infty$ is a subalgebra of $A_W$, and thus $(W,A)$ is a Schur pair giving rise 
to the algebra of commuting differential operators
\begin{equation}
\label{Schur-to-CAode}
\{U(x,\partial_x)f(\partial_x)U(x,\partial_x)^{-1}: f(z)\in A\}.
\end{equation}
The Baker--Akhiezer function defined in the above way coincides with the one from \S \ref{2.1}.
The construction of the above algebra of commuting differential operators 
from the data (i--iii) is called the {\bf{Krichever correspondence}}. Under it, KP flows correspond to orbits of quadruples $(X,\infty,z,D)$ under the natural action by the Jacobian of $X$.
As such, the associated solutions of the KP hierarchy naturally have expressions in terms of algebraic functions on $X$.

Mumford \cite{Mumford} extended the Krichever correspondence to singular projective curves.
A Krichever quintuple $(X,\mathcal L,\infty, t,\varphi)$ consists of the following data:
\begin{enumerate}
\item[(i')] A projective curve $X$; 
\item[(ii')] A maximally torsion free invertible sheaf $\mathcal L$ on $X$; 
\item[(iii')] A smooth point $\infty$ on $X$; 
\item[(iv')] An isomorphism $t : \mathbb D(\epsilon) := \{z\in\bbc: |z|<\epsilon\} \rightarrow U_\infty$ to a neighborhood $U_\infty$ of $\infty$;
\item[(v')] An $\mathcal O_{U_\infty}$-module isomorphism $\varphi: \mathcal L_{U_\infty}\rightarrow t_* \mathcal O_{\mathbb D}(-1)$.
\end{enumerate}
A Krichever quintuple defines a unique point $W\in\Gr_+(0)$ defined by 
\[
W = t^* \varphi(\mathcal L(U_\infty\diff\{\infty\}))
\]
along with a subalgebra $A$ of $A_W$ by 
\[
A = t^*\mathcal O_X(U_\infty \diff\{\infty\})
\]
to which we associate the algebra of commuting differential operators \eqref{Schur-to-CAode}.

The {\em{rank}} of a commutative algebra of differential operators is defined as the 
greatest common divisor of the orders of its elements.

\begin{thm}[Krichever \cite{Krichever1}, Mumford \cite{Mumford}]\label{ktheorem} The above constructions provides a
bijection between Krichever quintuples and rank one algebras of commuting differential operators.
\end{thm}

\section{Extending Burchnall--Chaundy to fractional differential operators}
\label{bcsect}
In this section we first give a new proof of the Burchnall--Chaundy Theorem \ref{bctheorem} that two commuting differential operators must satisfy an algebraic relationship.
The proof is simpler than the original proof and does not rely on actions of the operators on their respective kernels, which 
are unavailable in the case of fractional differential operators.
Consequently, we use this approach to obtain an extension of the Burchnall--Chaundy theory to fractional differential operators.
\subsection{A new approach to the Burchnall--Chaundy theorem}
\label{3.1}
\begin{proof}[A second proof of the Burchnall--Chaundy theorem \ref{bctheorem}]
Let $ P(x,\partial_x)$ and $ Q(x,\partial_x)$ be two commuting monic differential operators of orders $\ell$ and $m$, respectively.
For $N\geq 1$, let $S_N = \{ P(x,\partial_x)^i Q(x,\partial_x)^j: 0\leq i,j\leq N\}$.
Also for each $0\leq n \leq (\ell+m)N$ choose $ R_n(x,\partial_x)\in \vspan_\bbc S_N$ such that $ R_n(x,\partial_x)$ is monic of degree $n$ if it exists and $ R_n(x,\partial_x) =0$ otherwise.

Let $V = \vspan_\bbc\{ R_n(x,\partial_x): 0\leq n\leq (\ell+m)N\}$.
We claim that $V=\vspan_{\bbc}S_N$.
The inclusion $V\subseteq \vspan_{\bbc}S_N$ is obvious.
To prove the opposite direction, assume that $V\neq \vspan_\bbc S_N$.
Choose $ R(x,\partial_x)\in\vspan_{\bbc}S_N\diff V$ of smallest possible order $k$ and note that $k\leq (\ell+m)N$.
Since $ R(x,\partial_x)$ commutes with the monic operator $ P(x,\partial_x)$, it must have constant leading coefficient $\alpha\in\bbc\diff\{0\}$.
But then $ R(x,\partial_x)/\alpha$ is monic, so $ R_k(x,\partial_x)\neq 0$ and $ R(x,\partial_x)/\alpha - R_k(x,\partial_x)$ has order smaller than $k$.
By the minimality of $k$, it  follows that $ R(x,\partial_x)/\alpha - R_k(x,\partial_x)\in V$, but then $ R(x,\partial_x)\in V$, which is a contradiction.
This proves our claim.

As a consequence of the previous paragraph, we see that $\dim\vspan_\bbc S_N \leq \dim V \leq (\ell+m)N +1$.
Now if $ P(x,\partial_x)$ and $ Q(x,\partial_x)$ are algebraically independent, then the dimension of $\vspan_\bbc S_N$ is $(N+1)^2$, which grows quadratically with $N$.
Thus $ P(x,\partial_x)$ and $ Q(x,\partial_x)$ must be algebraically dependent.
\end{proof}

\subsection{An extension of the Burchnall--Chaundy theorem to fractional differential operators}
\label{3.2}
A key ingredient in the proof above is the upper bound on the order $\ord(F(P, Q))$ of a polynomial in $ P(x,\partial_x), Q(x,\partial_x)$.
In order to prove our extension of Burchnall and Chaundy's theorem to fractional differential operators, we need to also have a \emph{lower} bound on $\ord(F( P(x,\partial_x), Q(x,\partial_x)))$.
In the case that $ P(x,\partial_x)$ and $ Q(x,\partial_x)$ are differential operators, the orders are bounded below by $0$, so no other lower bound is necessary.
However, for fractional differential operators $F( P(x,\partial_x), Q(x,\partial_x))$ can have negative order, even if both the order of $ P(x,\partial_x)$ and $ Q(x,\partial_x)$ are positive
(for example, one can take $P(x,\partial_x) = \partial_x$, $ Q(x,\partial_x) = \partial_x + \partial_x^{-1}$ and $F(z,w) = z-w$).

To begin, we prove a lemma for finding common denominators of sequential products of fractional differential operators.
\begin{lem}\label{sequential product lemma}
Let $ P_i(x,\partial_x)$ and $Q_i(x,\partial_x)$ be differential operators for $1\leq i\leq n$ with $Q_i(x,\partial_x)\neq 0$ a nonzero differential operator of order $\ell_i$ for all $i$.
Then there exist differential operators $ L_\ell(x,\partial_x)$ and $ L_r(x,\partial_x)$ of order at most $\ell_1+\dots+\ell_n$ satisfying
$$ L_\ell P_1 Q_1^{-1} P_{2} Q_{2}^{-1}\dots P_k Q_k^{-1}\ \ \text{is a differential operator for all $1\leq k\leq n$},$$
$$ P_k Q_k^{-1} P_{k-1} Q_{k-1}^{-1}\dots P_1 Q_1^{-1} L_r\ \ \text{is a differential operator for all $1\leq k\leq n$}.$$
\end{lem}
\begin{proof}
We will prove the existence of $L_r(x,\partial_x)$, since the proof for $ L_\ell(x,\partial_x)$ is similar.
First note that for any differential operators $ P(x,\partial_x),  Q(x,\partial_x)$ there exist differential operators $ R(x,\partial_x), L(x,\partial_x)$ such that $ QR =  PL$ and $\ord( Q) \geq \ord( L)$.
To see this, take $ L^*$ to be a differential operator with $\ker( L^*) =  P^*\cdot\ker( Q^*)$.  Then $ Q^*$ right divides $ L^* P^*$, i.e. there exists $ R^*$ such that $ R^* Q^* =  L^* P^*$, so that $ PL=QR$ and $\ord(L) = \ord(L^*) = \dim\ker(L^*)\leq \dim\ker(Q^*) = \ord(Q^*) = \ord(Q)$.

Now define $ R_j,  L_j$ inductively by $ L_1 =  Q_1$, $ R_1 = 1$ and $ Q_{k+1} R_{k+1} =  P_k R_k L_{k+1}$ with $\ord( L_{k+1})\leq \ell_{k+1}$ for all $k\geq 0$.
Then in particular $ R_{k+1} L_{k+1}^{-1} =  Q_{k+1}^{-1} P_k R_k$ so that $L_r =  L_1 L_2\dots L_n$ satisfies
$$ P_k Q_k^{-1} P_{k-1} Q_{k-1}^{-1}\dots P_1 Q_1^{-1} L_r =  P_k R_k L_k L_{k+1}\dots  L_n$$
is a differential operator for all $1\leq k\leq n$.
\end{proof}

Now to get our lower bound on the order of polynomials in fractional differential operators, we introduce the notion of the \vocab{denominatorial order} of a fractional differential operator $ P(x,\partial_x)$, defined by
$$\dord( P) = \min\{\ord( L):  L\ \text{and}\ P L\ \text{are both differential operators with}\  L\neq 0\}.$$
In other words $\dord( P)$ is the minimal order of a right denominator for $ P$.
With this in mind, we have a lower bound defined by the following lemma.
\begin{lem}
Let $N>0$ be an integer and let $ P(x,\partial_x)$ and $ Q(x,\partial_x)$ be commuting fractional differential operators and let $F(z,w)\in\bbc[z,w]$ be a polynomial with $\deg_z(F)\leq N$ and $\deg_w(F)\leq N$.
Then if $F( P, Q)\neq 0$ we must have
$$-N(\dord( P) + \dord( Q)) \leq \ord(F( P, Q))\leq N(\ord( P) +\ord( Q)).$$
\end{lem}
\begin{proof}
The inequality $\ord(F( P, Q))\leq N(\ord( P) +\ord( Q))$ is obvious, so we just need to show the remaining inequality.
By the previous lemma, we can choose differential operators $ L_\ell$ and $ L_r$ of order at most $N\dord( P)$ and $N\dord( Q)$, respectively, such that $ L_\ell  P^j$ and $ Q^k L_r$ are differential operators for all $0\leq j,k,\leq N$.
It follows that $ L_\ell F( P, Q) L_r$ is a differential operator, and hence $\ord(F( P, Q)) \geq -(\ord( L_\ell) + \ord( L_r))\geq -N(\dord( P) + \dord( Q))$.
\end{proof}

With this lemma in place, we can prove our extension of Burchnall and Chaundy's theorem for fractional differential operators.
\begin{thm}
Let $ P(x,\partial_x)$ and $ Q(x,\partial_x)$ be commuting fractional differential operators, where $P$ is monic of nonzero order and $Q$ is not a constant.
Then $ P(x,\partial_x)$ and $ Q(x,\partial_x)$ are algebraically dependent.
\end{thm}
\begin{proof}
For $N\geq 1$, let $S_N = \{ P(x,\partial_x)^i Q(x,\partial_x)^j: 0\leq i,j\leq N\}$ and let $r = \ord( P(x,\partial_x)) + \ord( Q(x,\partial_x)) + \dord( P(x,\partial_x)) + \dord( Q(x,\partial_x))$.
By the previous lemma, we know that the differential operators in $\vspan_\bbc S_N$ will have orders between $-rN$ and $rN$.
For each integer $n$ with $|n|\leq rN$ choose $ R_n(x,\partial_x) \in \vspan_\bbc S_N$ such that $ R_n(x,\partial_x)$ is monic of degree $n$ if it exists and $ R_n(x,\partial_x) =0$ otherwise.

Let $V = \vspan_\bbc\{ R_n(x,\partial_x) : |c|\leq rN\}$.
We claim that $V=\vspan_{\bbc}S_N$.
The inclusion $V\subseteq \vspan_{\bbc}S_N$ is obvious.
To prove the opposite direction, assume that $V\neq \vspan_\bbc S_N$.
Choose $ R(x,\partial_x)\in\vspan_{\bbc}S_N\diff V$ of smallest possible order $k$ and note that $k\leq (\ell+m)N$.
Since $ R(x,\partial_x)$ commutes with the monic operator $ P(x,\partial_x)$, it must have constant leading coefficient $\alpha\in\bbc\diff\{0\}$.
But then $ R(x,\partial_x)/\alpha$ is monic, so $R_k(x,\partial_x)\neq 0$ and $ R(x,\partial_x)/\alpha - R_k(x,\partial_x)$ has order smaller than $k$.
By the minimality of $k$, it  follows that $ R(x,\partial_x)/\alpha - R_jk(x,\partial_x) \in V$, but then $ R(x,\partial_x)\in V$, which is a contradiction.
This proves our claim.

As a consequence of the previous paragraph, we see that $\dim\vspan_\bbc S_N \leq \dim V \leq 2rN +1$.
Now if $ P(x,\partial_x)$ and $ Q(x,\partial_x)$ are algebraically independent, then the dimension of $\vspan_\bbc S_N$ is $(N+1)^2$, which grows quadratically with $N$.
Thus $ P(x,\partial_x)$ and $ Q(x,\partial_x)$ must be algebraically dependent.
\end{proof}

\section{The spectral field of a point in Sato's Grassmannian}
The spectral algebra $A_W$ of a point $W$ in Sato's Grassmannian plays a fundamental role in integrable systems.  In particular, when $A_W$ is nontrivial 
(i.e. $\bbc\subsetneq A_W$) it provides a connection between algebraic geometry and commutative algebras of differential operators.  
However, $A_W$ is trivial (i.e. $A_W=\bbc$) for many points $W$ of $\Gr$, which limits the applications of this invariant.

In this section we introduce and study a much richer invariant of the planes in Sato's Grassmannian, the spectral field $K_W$ of $W\in\Gr$.
The spectral field $K_W$ is a natural extension of $A_W$ in the sense that it is the fraction field of $A_W$ when $A_W$ is nontrivial 
(see Proposition \ref{AW vs KW} below). However, as shown by Example \ref{trivial AW example}, $K_W$ can be nontrivial even when $A_W=\bbc$.
In Theorem \ref{KW and fractional operators} we use this new invariant to give a classification of maximal 
algebras of commuting fractional differential operators in terms of Sato's Grassmannian.
\subsection{The spectral field $K_W$}
\label{4.1}
Let $W\in\Gr$.
We define the \vocab{spectral field} $K_W$ of $W$ to be
$$K_W = \{f(z)\in\mathbb L: \dim (W + f(z)W)/W < \infty\}.$$
It is clear from the definition that $A_W\subseteq K_W$. It is a nontrivial fact that $K_W$ is always a field. This is proved in Proposition \ref{Kw-field}. 
For its proof we will need the following lemma.
\begin{lem}
\label{Kw-second-fin} For all $W \in \Gr$ and $f(z) \in K_W$, 
\[
\dim (W /(f(z) W \cap W))  < \infty.
\]
\end{lem}
\begin{proof} For a subspace $U$ of $\bbl$ and $n \in \bbz$, set
\[
U_n = \{ u(z) \in U : \deg u(z) \leq n \}.
\]
Since $W \in \Gr$, there exist $c, n_0 \in \bbz$ such that  
\begin{equation}
\label{stab-dim}
\deg W_n = n+c \quad \mbox{for all} \quad n \geq n_0.
\end{equation}
By the assumption that $f(z) \in K_W$, we may choose a finite dimensional subspace $E \subset f(z) W$ 
such that
\[
f(z)W + W = E \oplus W. 
\]
Set 
\[
d= \deg f(z) \quad \mbox{and} \quad
\ell = \max \{ \deg e(z) : e(z) \in E \}.
\]
From the bijectivity of the multiplication by $f(z)$ on $\bbl$ we have 
\[
E \subset f(z) W_n = (f(z) W)_{n+d}, \; \; \mbox{and thus}, \; \;
(f(z) W)_{n+d} = ((f(z) W)_{n+d} \cap W_{n+d}) \oplus E
\]
for $n \geq \ell - d$.
This implies that
\begin{equation}
\label{dim-ineq}
\dim ( (f(z) W)_{n+d} \cap W_{n+d})= \dim (f(z) W)_{n+d} - \ell= \dim W_n - \ell \quad \mbox{for} \quad n \geq \ell - d.
\end{equation}
Combining \eqref{stab-dim} and \eqref{dim-ineq} gives that for $n \geq \max \{ n_0, n_0 -d, \ell-d \}$, 
\begin{align*}
\dim W_{n+d} - \dim ( (f(z) W \cap W)_{n+d}) 
&= n+d +c - (n+c - \ell) 
\\
&= d-\ell. 
\end{align*}
Therefore, $\dim (W /(f(z) W \cap W)) = d -\ell$.
\end{proof}

\begin{prop}
\label{Kw-field}
The set $K_W$ is a field.
\end{prop}
\begin{proof}
Suppose that $f_1(z),f_2(z)\in K_W$. For $i=1,2$,
choose finite dimensional subspaces $E_i' \subset f_i(z) W$ such that $f_i(z)W+W = W \oplus E_i'$. 
We have
\begin{align*}
&\dim \frac{W + (f_1(z)+f_2(z))W}{W} \leq \dim \frac{W + f_1(z)W + f_2(z)W}{W} \leq \dim (E_1'+E_2') < \infty,\\
&\dim \frac{W + f_1(z)f_2(z)W}{W} \leq \dim \frac{W + f_1(z)(W\oplus E_2')}{W} \leq \dim (E_1'+ f_1(z)E_2') < \infty.
\end{align*}
Therefore $f_1(z)+f_2(z), f_1(z)f_2(z) \in K_W$.

Using the second isomorphism theorem for abelian group homomorphisms, Lemma \ref{Kw-second-fin}
and the bijectivity of the $f(z)$ multiplication on $\bbl$, we obtain
\[
(f_i(z)^{-1} W + W)/ W \cong (f_i(z)^{-1} W)/ ( (f_i(z)^{-1} W) \cap W) \cong W/ ( W \cap f_i(z) W) < \infty.
\]
Hence $f_i(z)^{-1} \in K_W$, and thus, $K_W$ is a field. 
\end{proof}

When $A_W$ is nontrivial we can identify $K_W$ with the fraction field of $A_W$.
\begin{prop}\label{AW vs KW}
If $\bbc\subsetneq A_W$ then $K_W$ is the fraction field of $A_W$.
\end{prop}
\begin{proof}
Since $A_W\subseteq K_W$ and $K_W$ is a field, we know that the fraction field of $A_W$ will necessarily be a subset of $K_W$.

Let $g(z)\in A_W\diff\bbc$ and suppose $f(z)\in K_W$.
Since $g(z)W\subseteq W$, $g(z)$ induces an endomorphism of $W + f(z)W$ which descends to an endomorphism of the quotient space $(W + f(z)W)/W$.
The quotient space is finite dimensional, so there exists a nonzero polynomial $h(z)$ with the property that $h(g(z))$ acts trivially.
Thus 
\[
h(g(z))W + f(z)h(g(z))W\subseteq W, 
\]
implying that $f(z)h(g(z))W\subseteq W$ so that $f(z)h(g(z))\in A_W$.
Note that $\deg(g(z))>0$ so $h(g(z))$ is a nonzero element of $A_W$.  Thus $f(z)$ is in the fraction field of $A_W$,
which proves the proposition.
\end{proof}

We should not interpret the previous proposition to mean that $K_W$ is nothing more than the fraction field of $A_W$ always.
Indeed, $K_W$ can be nontrivial even when $A_W$ is trivial.
\begin{ex}\label{trivial AW example}
Consider the point $W\in\Gr_+(0)$ defined by
\begin{align*}
W & = \text{span}\{z^n + nz^{n-2}: n\geq 0\}\\
	& = \text{span}\left\lbrace z^{2n}:n\geq 0\right\rbrace \oplus \text{span}\left\lbrace z^{2n+1} + (-1)^n\frac{(2n+1)!!}{z}: n\geq 0\right\rbrace.
\end{align*}
Suppose that $f(z)\in A_W$, i.e. that $f(z)W\subseteq W$.
Since $1\in W$, it follows that $f(z)\in W$. Furthermore, $zW\subseteq\mathbb L_+$ so $f(z) = \sum_{j=-1}^m a_jz^j$ for some $a_j \in \bbc$ for $- 1\leq j \leq m$.
Thus $f(z)z^{2n} = \sum_{j=-1}^m a_jz^{j+2n}$ must be in $W$
implying that 
\[
\sum_{k=0}^{\lfloor m/2\rfloor} (-1)^{n+k}a_{2k-1} (2n+2k-1)!! =0, \quad \forall n \in \bbn.
\]
It follows that $a_{2j-1} = 0$ for all $0 \leq j \leq \lfloor m/2\rfloor$.
Similarly, by considering expressions of the form $f(z)(z^{2n+1} + (-1)^n(2n+1)!!z^{-1})$, we obtain that $a_{2 j}=0$ for all $0 < j \leq \lfloor m/2\rfloor$. 
Thus $f(z)=a_0$ must be a constant and $A_W = \bbc$.
Moreover, for all polynomials $f(z)$
$$\dim\frac{W + f(z)W}{W} \leq \dim\frac{z^{-1}\mathbb L_+}{W} = 1.$$
Hence $\bbc[z]\subseteq K_W$ and so $\bbc(z)\subseteq K_W$.
\qed
\end{ex}

\subsection{Connection with maximal algebras of commuting differential operators}
\label{4.2}
As explained previously, the spectral algebra $A_W$ of a point $W$ of Sato's Grassmannian is intimately connected with a commutative algebra of differential operators.
In this section, we extend this connection to a connection between the spectral fields $K_W$ and maximal algebras of commuting fractional differential operators.

To begin, we recall an important characterization of the differential operators inside the ring of pseudodifferential operators.
Its proof is standard and is omitted for brevity.
\begin{lem}
\label{descr-do}
Let $L(x,\partial_x)$ be a pseudodifferential operator.
Then $L(x,\partial_x)$ is a differential operator if and only if $\mathbb L_+\cdot L(x,\partial_x)\subseteq \mathbb L_+$.
\end{lem}
In Lemma \ref{char-fdo} we obtain an extension of this result to a characterization of fractional differential operators.
To prove it, we first require a preliminary result.
\begin{lem}  
\label{exists-do}
Suppose that $E\subseteq\mathbb L$ is finite dimensional.
Then there exists a monic differential operator $Q(x,\partial_x)$ satisfying $E\cdot Q(x,\partial_x)\subseteq \mathbb L_+$.
\end{lem}
\begin{proof}
It suffices to show that for any integer $d>0$ and $f(z)\in\mathbb L$ of the form $f(z) = z^{-d} + \sum_{n=d+1}^\infty a_nz^{-n}$ there exists a monic differential operator $Q(x,\partial_x)$ with $f(z)\cdot Q(x,\partial_x)\in\mathbb L_+$.
Then, by using a product differential operators we can send a basis of $E$ into $\mathbb L_+$ and thus send $E$ itself into $\mathbb L_+$.

Consider an arbitrary  $q(x) = \sum_{k=0}^\infty q_kx^k$. Then 
$$f(z)\cdot q(x) = z^{-d}q_0 + \sum_{\ell=d+1}^\infty\left(\sum_{m=0}^{\ell-d} (-1)^m m!\binom{\ell-1}{m}a_{\ell-m}q_m\right)z^{-\ell},$$
where we set $a_d=1$.
If we choose $q_0=1$ and
$$q_{\ell-d} = \sum_{m=0}^{\ell-d-1}(-1)^{\ell-d-m-1}\frac{(d-1)!}{(\ell-1-m)!}a_{\ell-m}q_m,$$
then $q(x)$ is a unit in $\bbc[[x]]$ and $f(z)\cdot q(x) = z^{-d}$.
Thus $Q(x,\partial_x) := q(x)\partial_x^dq(x)^{-1}$ is a monic differential operator satisfying $f(z)\cdot Q(x,\partial_x)\in\mathbb L_+$.
\end{proof}
We now obtain our characterization of fractional differential operators.
\begin{lem}
\label{char-fdo}
Let $L(x,\partial_x)$ be a pseudodifferential operator.
Then $L(x,\partial_x)$ is a fractional differential operator if and only if
$$\dim(\mathbb L_+ +  \mathbb L_+\cdot L(x,\partial_x))/\mathbb L_+ <\infty.$$
\end{lem}
\begin{proof}
First, suppose that $L(x,\partial_x)$ is a fractional differential operator.
Then $L(x,\partial_x) = L_2(x,\partial_x)^{-1}L_1(x,\partial_x)$ for some differential operators $L_1$ and $L_2$.
Since $\mathbb L_+\cdot L_2(x,\partial_x)\subseteq \mathbb L_+$, there is a subspace $E_2\subseteq \mathbb L_+$ satisfying $\mathbb L_+\cdot L_2(x,\partial_x)\oplus E_2 = \mathbb L_+$.
By order arguments, $E_2$ is finite dimensional of dimension less than or equal to the order of $L_2$.
Using that $\mathbb L_+\cdot L_2(x,\partial_x)^{-1} = \mathbb L_+ \oplus E_2 \cdot L_2(x,\partial_x)^{-1}$ 
and Lemma \ref{descr-do}, we obtain
\begin{align*}
&\dim(\mathbb L_+ + \mathbb L_+\cdot L(x,\partial_x))/\mathbb L_+ 
\leq \dim(\mathbb L_+ + \mathbb L_+\cdot L_1(x,\partial_x) + E_2  L_1(x,\partial_x) )/\mathbb L_+ 
\\
&= \dim(\mathbb L_+ + E_2  L_1(x,\partial_x) )/\mathbb L_+ 
\leq \dim (E_2 L_1(x,\partial_x))
\leq \dim E_2 < \infty.
\end{align*}

Conversely, assume $L(x,\partial_x)$ is a pseudodifferential operator satisfying
$$\dim(\mathbb L_+ +  \mathbb L_+\cdot L(x,\partial_x))/\mathbb L_+ <\infty,$$
and let $E\subseteq \mathbb L$ be a finite dimensional subspace satisfying $\mathbb L_+ + \mathbb L_+\cdot L(x,\partial_x) = \mathbb L_+ \oplus E$.
By the previous lemma, we can choose a monic differential operator $Q(x,\partial_x)$ satisfying $E\cdot Q(x,\partial_x)\subseteq \mathbb L_+$.
Hence $\mathbb L_+\cdot L(x,\partial_x)Q(x,\partial_x) \subseteq \mathbb L_+$, and by Lemma \ref{descr-do}, $L(x,\partial_x)Q(x,\partial_x)$ a differential operator.
Therefore $L(x,\partial_x)$ is a fractional differential operator.
\end{proof}
\begin{lem}
\label{biject}
Consider $W\in\Gr_+(0)$ with $W = \mathbb L_+\cdot U(x,\partial_x)$ for a pseudodifferential operator $U(x,\partial_x) = 1 + \sum_{n=1}^\infty u_n(x)\partial_x^{-n}$.
For $f(z) \in \bbl$, 
\[
U(x,\partial_x)f(\partial_x)U(x,\partial_x)^{-1}
\]
is a fractional differential operator if and only if $f(z) \in K_W$. 
\end{lem}
\begin{proof} Denote 
\[
L(x, \partial_x) = U(x,\partial_x)f(\partial_x)U(x,\partial_x)^{-1}.
\]
We have 
\begin{align*}
\dim(\mathbb L_+ + \mathbb L_+\cdot L(x,\partial_x))/\mathbb L_+
  & = \dim(\mathbb L_+ + W\cdot f(\partial_x)U(x,\partial_x)^{-1})/\mathbb L_+\\
  & = \dim(W + W\cdot f(\partial_x))/W\\
  & = \dim(W + f(z)W)/W.
\end{align*}
The stated equivalence now follows from the characterization of fractional differential operators in Lemma \ref{char-fdo}.
\end{proof} 
The next theorem provides a classification of all maximal algebras of commuting fractional differential operators.
\begin{thm}\label{KW and fractional operators}
Let $W\in\Gr_+(0)$ with $W = \mathbb L_+\cdot U(x,\partial_x)$ for a pseudodifferential operator $U(x,\partial_x) = 1 + \sum_{n=1}^\infty u_n(x)\partial_x^{-n}$.
Then the algebra
\[
{\mathscr{K}}_W=\{U(x,\partial_x)f(\partial_x)U(x,\partial_x)^{-1}: f(z)\in K_W\}
\]
is a commutative algebra of fractional differential operators and is maximal if $K_W\neq \bbc$.

All maximal commutative algebras of fractional differential operators containing a normalized fractional differential operator $L(x, \partial_x)$ of nonzero order
arise in this way.
\end{thm}
\begin{proof}
The algebra ${\mathscr{K}}_W$ is obviously commutative. It follows from Lemma \ref{biject} that 
it consists of fractional differential operators.

If $L'(x,\partial_x)$ is another fractional differential operator commuting with $L(x,\partial_x)$, then 
\[
L'(x,\partial_x) =  U(x,\partial_x)f(\partial_x)U(x,\partial_x)^{-1}
\]
for some $f(z) \in \bbc((z))$. Applying Lemma \ref{biject}, we obtain that $f(z) \in K_W$.
Therefore $L'(x,\partial_x) \in {\mathscr{K}}_W$, and thus, ${\mathscr{K}}_W$ is a maximal commutative algebra of fractional differential operators.

Let ${\mathscr{K}}$ be a maximal commutative algebra of fractional differential operators containing a normalized fractional differential operator $L(x, \partial_x)$ of order $m\neq 0$. 
By Schur's theorem \cite{Schur},
\[
L(x, \partial_x) = U(x,\partial_x) \partial_x^m U(x,\partial_x)^{-1} 
\]
for some pseudodifferential operator $U(x,\partial_x) = 1 + \sum_{n=1}^\infty u_n(x)\partial_x^{-n}$. Therefore, 
${\mathscr{K}}$ is contained in the centralizer of $L(x, \partial_x)$ in the algebra of fractional differential operators, 
which by the first part of the proof is ${\mathscr{K}}_W$ for $W = \bbl_+ \cdot U(x, \partial_x)$.
Since ${\mathscr{K}}_W$ is a maximal commutative algebra of fractional differential operators, 
${\mathscr{K}} = {\mathscr{K}}_W$. 
\end{proof}

\section{Krichever correspondence for fractional differential operators}
In this section we present an extension of the Krichever correspondence from Theorem \ref{ktheorem} to algebras of 
commuting fractional differential operators of rank 1.

Throughout this section, we adopt the following notation.
Let $X$ be an irreducible, reduced algebraic curve (not necessarily projective).
Fix a smooth point $p\in X$ and let $A$ be a $\mathbb C$-algebra with $\spec(A) = X\diff \{p\}$.
If $z^{-1}$ is a uniformizer for the complete local ring $\wh{\mathcal O_{X,p}}$, and $\varphi: \Gamma(U,\mathcal L)\cong \Gamma(U,\mathcal O_{X})$ is a local trivialization of $\mathcal L$ in a neighborhood $U$ of $p$, then the composition
$$\Gamma(X\diff\{p\},\mathcal L)\rightarrow \Gamma(U\diff\{p\},\mathcal L)\xrightarrow{\varphi}\Gamma(U\diff\{p\},\mathcal O_X)\rightarrow (\mathcal O_p)_{\mathfrak m_p} \rightarrow (\wh{\mathcal O_{X,p}})_{\wh{\mathfrak m_p}} = \bbc((z^{-1}))$$
identifies the $A$-module $\Gamma(X\diff\{p\},\mathcal L)$ with a subspace $W\subseteq\mathbb L= \mathbb C((z^{-1}))$ which is a point of Sato's Grassmannian $\Gr$.
As an abuse of notation, we will write $\varphi: \Gamma(X\diff\{p\},\mathcal L)\rightarrow W$ for the isomorphism determined by this composition.
The algebra $A$ is identified via the restriction
$$\Gamma(X\diff\{p\},\mathcal O)\rightarrow\Gamma(U\diff\{p\},\mathcal O_X)\rightarrow (\mathcal O_{X,p})_{\mathfrak m_p} \rightarrow (\wh{\mathcal O_{X,p}})_{\wh{\mathfrak m_p}} = \bbc((z^{-1}))$$
with a subalgebra of $A_W = \{f(z)\in\mathbb L: f(z)W\subseteq W\}$.
In this way, we obtain a Schur pair $(W,A)$ from the Krichever quintuple $(X,\mathcal L,p,z^{-1},\varphi)$.
\begin{remk}
	Geometrically, the localization of the completion of the stalk of the structure sheaf at $p$, $(\wh{\mathcal O_{X,p}})_{\wh{\mathfrak m_p}}$, may be thought of as a ring of algebraic functions on a formal annular neighborhood $Z_p = \spec((\wh{\mathcal O_{X,p}})_{\wh{\mathfrak m_p}})$ centered at $p$, which we refer to later as the \vocab{formal annulus at $p$}.
	Notationally, for any $\mathcal O_X$-module $\mathcal F$ we will use $\Gamma(Z_p,\mathcal F)$ to denote global sections of the pullback of $\mathcal F$ to $Z_p$, i.e. $\Gamma(Z_p,\mathcal F) = \Gamma(X\diff\{p\},\mathcal F)\otimes_A (\wh{\mathcal O_{X,p}})_{\wh{\mathfrak m_p}}$.
\end{remk}

To extend Krichever's construction to commuting algebras of fractional differential operators, we will work with a hextuple of data, consisting of a Krichever quintuple plus a section of $\chi$ of the dual of the semi-infinite jet bundle of $\mathcal L$ over the formal annulus $(\wh{\mathcal O_{X,p}})_{\wh{\mathfrak m_p}}$ at $p$ which extends to a section of the dual of $\mathcal F$ over $X\diff\{p\}$ for some subsheaf $\mathcal F$ of the semi-infinite jet bundle with finite codimension.
The corresponding point in Sato's Grassmannian is given by the image of the sequence of maps
$$\Gamma(X\diff\{p\},\mathcal L)\rightarrow \Gamma(U\diff\{p\},\mathcal L)\rightarrow\Gamma(Z_p,\mathcal L)\xrightarrow{\chi\circ j^\infty}\Gamma(Z_p,\mathcal L) \xrightarrow{\varphi} (\wh{\mathcal O_{X,p}})_{\wh{\mathfrak m_p}} = \mathbb C((z^{-1})).$$

\subsection{Duals of finite jet bundles}
We begin with a more detailed look at duals of bundles of finite jets.
Note that sections of the bundle of $m$-jets over $X\diff\{p\}$ may be expressed in terms of the Schur pair $(W,A)$ above by
$$\Gamma(X\diff\{p\},\mathcal J^n(\mathcal L))\cong (A\otimes_{\mathbb C} W)/I_\Delta^{n+1}(A\otimes_{\mathbb C} W)$$
where here $I_\Delta = \langle \{a\otimes 1-1\otimes a: a\in A\}\rangle$ and the action of $A$ is defined on simple tensors by $a (b\otimes s) = ab\otimes s$.

As in \eqref{linear operator formulation of dual}, sections of the dual $\mathcal J_n(\mathcal L)$ of $\mathcal J^n(\mathcal L)$ correspond to linear differential operators $\mathcal L\rightarrow\mathcal L$ of order $n$.
To make this correspondence very explicit, it is useful to recall Grothendieck's characterization of differential operators in terms of $\Ad$-vanishing linear transformations.
\begin{defn}
Let $M$ be an $R$-module for a commutative $\mathbb C$-algebra $R$ and let $\psi: M\rightarrow M$ be a $\mathbb C$-linear transformation.
Then $M$ is \vocab{$\Ad$-vanishing of order $\leq n$} if
$$[\lambda_{a_1},[\lambda_{a_2},\dots[\lambda_{a_n},[\lambda_{a_{n+1}},\psi]]\dots]]= 0$$
for any sequence $a_1,\dots, a_{n+1}\in R$, where here $\lambda_a: M\rightarrow M$ denotes the linear transformation $\lambda_a(m) = am$ and $[\cdot,\cdot]$ denotes the commutator bracket.
\end{defn}

The following lemma characterizes $\Ad$-vanishing operators as differential operators.
\begin{lem}
Suppose that $\psi$ is a $\bbc$-linear transformation of $W$ which is $\Ad$-vanishing of order $\leq n$.
Then there exist $a_0(z),a_1(z),\dots, a_n(z)\in\mathbb L$ such that 
$$\psi(v(z)) = \sum_{k=0}^n a_k(z)\frac{\partial^k v(z)}{\partial z^k}$$
for all $v(z)\in W$.
\end{lem}
\begin{proof}
Clearly any linear differential operator of order $\leq n$ will be $\Ad$-vanishing of order $\leq n$, so we need only prove that all $\Ad$-vanishing linear transformations of order $\leq n$ are of this form.

We proceed by induction on $n$.
Suppose that $\psi$ is $\Ad$-vanishing of order $\leq n$.
If $n=0$, then $\psi$ is an $A$-module homomorphism $W\rightarrow W$, so $\psi(v(z)) = a_0(z)v(z)$ for some $a_0(z)\in A_W$.
Inductively, suppose that our statement is true for linear transformations which are $\Ad$-vanishing of order $\leq n-1$ and choose $a\in A\diff\mathbb C$.  Then $[\lambda_a,\psi]$ is $\Ad$-vanishing of order $\leq n-1$, so by assumption
$$[\lambda_a,\psi] = \sum_{k=0}^{n-1} b_k(z)\frac{\partial^k v(z)}{\partial z^k}$$
for some $b_0(z),\dots, b_{n-1}(z)\in \mathbb L$.
Then if we choose $a_0(z),\dots,a_n(z)\in\mathbb L$ to satisfy
$$b_j(z) = -\sum_{k=j+1}^n\binom{k}{j}a^{(k-j)}(z)a_k(z),$$
it follows that
$$\psi(v(z)) = \sum_{k=0}^n a_k(z)\frac{\partial^k v(z)}{\partial z^k},\quad\forall v(z)\in W.$$
\end{proof}

There is a simple bijective correspondence between sections of the dual of the bundle of $n$-jets and $\Ad$-vanishing linear operators of order $\leq n$, as shown by the next lemma.
\begin{lem}
Let $\psi: W\rightarrow W$ be an $\Ad$-vanishing $\bbc$-linear operator of order $\leq n$.
Then the function $\chi: (A\otimes_{\mathbb C} W)/I_\Delta^{n+1}(A\otimes_{\mathbb C} W)\rightarrow W$ defined on simple tensors by
$$\chi(a(z)\otimes v(z)) = a(z)\psi(v(z))$$
for all $v(z)\in W$ is an $A$-module homomorphism $\Gamma(X\diff\{p\},\mathcal J^n(\mathcal L))\rightarrow\Gamma(X\diff\{p\},\mathcal L)$.
Moreover, every element of $\Gamma(X\diff\{p\},\mathcal J_n(\mathcal L))$ is of this form.
\end{lem}
\begin{proof}
The $A$-module structure on $A\otimes_{\mathbb C}W$ compatible with the jet bundle structure is defined on simple tensors by $a(b\otimes v) = ab\otimes v$, so any linear transformation $\psi: W\rightarrow W$ defines an $A$-module homomorphism $\chi: A\otimes_{\mathbb C}W\rightarrow A\otimes_{\mathbb C}W$ defined on simple tensors by $\chi(a\otimes v)\mapsto a\psi(v)$.
Thus there is a natural bijective correspondence between linear transformations $\psi: W\rightarrow W$ and $A$-module homomorphisms $A\otimes_{\mathbb C}W\rightarrow W$.

To prove the equivalence between $\Ad$-vanishing linear transformations and sections of the dual of the jet bundle, notice that for any $a_1,\dots,a_n,b\in A$ and $v\in W$
$$\chi\left(\left(\prod_{k=1}^{n+1}(a_k\otimes 1-1\otimes a_k)\right)(b\otimes v)\right) = b[\lambda_{a_1},[\lambda_{a_2},\dots[\lambda_{a_n},[\lambda_{a_{n+1}},\psi]]\dots]](v).$$
Consequently an $A$-module homomorphism $\chi$ defined by a linear transformation $\psi$ descends to the quotient $(A\otimes_{\mathbb C}W)/I_\Delta^{n+1}(A\otimes_{\mathbb C}W)\rightarrow W$ if and only if $\chi$ is $\Ad$-vanishing of order $\leq n$.
\end{proof}

Putting the previous two lemmas together, we have the following characterization of sections of the dual of $\mathcal J^n(\mathcal L)$ over $X\diff\{p\}$.
\begin{thm}
There is a bijective correspondence $L\mapsto \chi_L$ between sections $\chi$ of $\mathcal J_n(\mathcal L)$ over $X\diff\{p\}$ and the set
$$\left\lbrace L(z,\partial_z) = \sum_{k=0}^n a_k(z)\frac{\partial^k}{\partial z^k}: a_0(z),\dots, a_n(z)\in \mathbb L,\ L(z,\partial_z)\cdot W\subseteq W\right\rbrace,$$
satisfying
$$\chi_L(j^nv) = L(z,\partial_z)v(z),\quad \forall v(z)\in W.$$
\end{thm}
\begin{proof}
This follows immediately from the previous two lemmas.
\end{proof}

\subsection{Dual of the semi-infinite jet bundle and pseudodifferential operators}
\label{5.2}
We now elaborate on the dual of the semi-infinite jet bundle and its connection with pseudodifferential operators.
As defined in \ref{def: weak jet bundle}, the semi-infinite jet bundle of $\mathcal L$ is defined as the colimit
$$\mathcal J^{\infty,0}(\mathcal L) = \varinjlim_n\mathcal J^n(\mathcal L)$$
with respect to the system of canonical injections $\iota_{m,n}: \mathcal J^m(\mathcal L)\rightarrow \mathcal J^n(\mathcal L)$ for $m<n$.

Since $\shom$ anticommutes with colimits in the first slot, it follows that the \vocab{dual of the semi-infinite jet bundle} over $\mathcal L$ is 
$$\shom_{\mathcal O_X}(\mathcal J^{\infty,0}(\mathcal L),\mathcal L) = \varprojlim_n\shom_{\mathcal O_X}(\mathcal J^n(\mathcal L),\mathcal L) = \varprojlim_n \mathcal J_n(\mathcal L) .$$

As we showed above, a section of $\mathcal J^n(\mathcal L)$ over $X\diff\{p\}$ is equivalent to a differential operator $L(z,\partial_z)$ of order $\leq n$ with coefficients in $\mathbb L$.
The projections $\mathcal J_n(\mathcal L)\rightarrow\mathcal J_\ell(\mathcal L)$ correspond to truncations of differential operators $\sum_{k=0}^na_k(z)\frac{\partial^k}{\partial z^k}\mapsto \sum_{k=0}^\ell a_k(z)\frac{\partial^k}{\partial z^k}$.
This gives us the following description of sections of the dual of the semi-infinite jet bundle.

\begin{thm}
There is a bijective correspondence $L\mapsto \chi_L$ between sections $\chi$ of the dual of $\mathcal J^{\infty,0}(\mathcal L)$ over $X\diff\{p\}$ and the set
\begin{equation}
\left\lbrace L(z,\partial_z) = \sum_{k=0}^\infty a_k(z)\partial_z^k: a_k(z)\in\mathbb L,\ \sum_{k=0}^\infty a_k(z)\frac{\partial^k v(z)}{\partial z^k}\in W\ \forall v(z)\in W\right\rbrace,
\end{equation}
where here $\chi_L$ is defined on $\infty$-jets by $\chi_L(j^{\infty} v) = L(z,\partial_z)\cdot v(z)$.
\end{thm}
\begin{remk}
For any $v(z),a_0(z),a_1(z),\dots\in \mathbb L$, the expression $\sum_{k=0}^\infty a_k(z)\frac{\partial^k v(z)}{\partial z^k}$ is well-defined, since the coefficient of $z^n$ is a sum of at most finitely many terms for all $n\in\mathbb Z$.
\end{remk}

\begin{remk}
Recalling the relationship between differential operators and $\Ad$-vanishing $\mathbb C$-linear transformations, we can realize a section of the dual of the semi-infinite jet bundle as a limit of $\Ad$-vanishing linear transformations, which is simply a linear transformation from $W$ to itself.
\end{remk}

The sections of $\mathcal J^{\infty,0}(\mathcal L)$ over the formal annulus $Z_p$ are simply Laurent series in $z^{-1}$, i.e. $\Gamma(Z_p,\mathcal J^{\infty,0}(\mathcal L))\cong \mathbb L$.
As a consequence, we see that the sections of the dual of $\mathcal J^{\infty,0}(\mathcal L)$ over the formal annulus $Z_p$ are given by differential operators of infinite order with coefficients in $\mathbb L$.

Now consider the ring $\mathcal P$ of pseudodifferential operators with coefficients in the ring of formal power series in $x$.  We can express an element $D(x,\partial_x)\in \mathcal P$ as 
$$D(x,\partial_x) = \sum_{k=-N}^\infty b_k(x)\partial_x^{-k} = \sum_{j=0}^\infty x^ja_j(\partial_x^{-1})$$
for some $b_k(x)\in\bbc[[x]]$ or $a_j(z^{-1})\in \bbc((z^{-1}))$.
To any such differential operator we associate the formal section of the dual of $\mathcal J^{\infty,0}(\mathcal L)$ defined on $\infty$-jets by
$$\chi_D(j^\infty v(z)) = \sum_{j=0}^\infty a_j(z)\frac{\partial^j v(z)}{\partial z^j}.$$
This is a section of the dual of $\mathcal J^{\infty,0}(\mathcal L)$ over the formal annulus $Z_p$, but in general is not a well-defined section over any affine open subset of $X$.

Now consider the case when $\mathcal L$ has trivial cohomology, so that $W=\Gamma(X\diff\{p\},\mathcal L)$ is an element of $\Gr_+(0)$.  Choose $Q(x,\partial_x)=1+ \sum_{k=1}^\infty b_k(x)\partial_x^{-k}$ satisfying $W = \mathbb L_+\cdot Q(x,\partial_x)$.  We can characterize the pseudodifferential operators that $Q(x,\partial_x)$ conjugates to differential operators as those corresponding to sections $\chi_D$ of the dual of $\mathcal J^{\infty,0}(\mathcal L)$ over the formal annulus which extend to sections over $X\diff\{p\}$.
\begin{thm}
Let $D(x,\partial_x) = \sum_{j=0}^\infty x^ja_j(\partial_x^{-1})\in \mathcal P$ be a pseudodifferential operator.
Then $L(x,\partial_x) = Q(x,\partial_x)D(x,\partial_x)Q(x,\partial_x)^{-1}$ is a differential operator if and only if there exists a section of $\chi$ of the dual of $\mathcal J^{\infty,0}(\mathcal L)$ over $X\diff\{p\}$ satisfying
$$\chi(j^\infty v) = \sum_{j=0}^\infty a_j(z)\frac{\partial^j v(z)}{\partial z^j}.$$
\end{thm}
\begin{proof}
Suppose $L(x,\partial_x)$ is a differential operator.
Then we know that $\mathbb L_+\cdot L(x,\partial_x)\subseteq \mathbb L_+$ and therefore $W\cdot D(x,\partial_x)\subseteq W$.
In particular, the section $\chi$ of the dual of $\mathcal J^{\infty,0}$ over $X\diff\{p\}$ defined by
$$\chi(j^\infty(v(z))) = v(z)\cdot D,\ \forall v(z)\in W$$
is well-defined.

Conversely, suppose that a section $\chi$ of the dual of $\mathcal J^{\infty,0}(\mathcal L)$ over $X\diff\{p\}$ satisfying the property stated in the theorem.
Then $\chi(j^{\infty}(v(z)))\in  W$ for all $v(z)\in W$ and therefore $W\cdot D(x,\partial_x)\subseteq W$.
It follows that $\mathbb L_+\cdot L(x,\partial_x)\subseteq \mathbb L_+$ so that $L(x,\partial_x)$ is a differential operator.
\end{proof}

As a consequence of this theorem, the pseudodifferential operators conjugating to fractional differential operators under $Q(x,\partial_x)$ are going to correspond to sections of the dual of $\mathcal J^{\infty,0}(\mathcal L)$ defined locally by
$$\chi(j^\infty(\chi_{D_1}(j^\infty(v(z))))) = \chi_{D_2}(j^\infty(v(z)))$$
for some pseudodifferential operators $D_1,D_2\in\mathcal P$ satisfying $W\cdot D_j\subseteq W$ for $j=1,2$.
Equivalently, these are pseudodifferential operators $D(x,\partial_x) = D_1(x,\partial_x)^{-1}D_2(x,\partial_x)$ satisfying $W'\cdot D(x,\partial_x)\subseteq W$ for some subspace $W'\subseteq W$ of finite codimension.
This is equivalent to an element of the dual of a subsheaf $\mathcal F\subseteq J^{\infty,0}(\mathcal L)$ of finite codimension.

\begin{defn}
\label{fract-sect}
A \vocab{rational section} of the dual of $\mathcal J^{\infty,0}(\mathcal L)$ over $U\subseteq X$ is a section $\chi$ of the dual of the semi-infinite jet bundle over the formal annulus $Z_p$, which extends to a section of the dual of $\mathcal F$ over $X\diff\{p\}$ for some submodule $\mathcal F$ of $\mathcal J^{\infty,0}(\mathcal L)$ of finite codimension.
\end{defn}

In algebraic geometry a rational morphism between schemes is a map from a dense open subset of the first scheme to the second one. From this point of view,  
a rational section of the dual of the jet bundle is 
the same thing as a rational morphism from the underlying scheme to the one associated to the scheme theoretic vector bundle. 

With this definition in mind, we have the following characterization of fractional differential operators.

\begin{thm}\label{cool catz}
Let $D(x,\partial_x)\in \mathcal P$ be a pseudodifferential operator.
Then $L(x,\partial_x) = Q(x,\partial_x)D(x,\partial_x)Q(x,\partial_x)^{-1}$ is a fractional differential operator if and only if there is a rational section 
$\chi$ of the dual of $\mathcal J^{\infty,0}(\mathcal L)$ over $X\diff\{p\}$ satisfying $\chi(j^\infty(v)) = \chi_D(j^\infty(v))$ for all $v(z)\in W$ with $j^\infty(v)$ in the domain of $\chi$.
\end{thm}
\begin{proof}
If $L(x,\partial_x)$ is a fractional differential operator, then there exist differential operators $L_1(x,\partial_x)$ and $L_2(x,\partial_x)$ satisfying $L(x,\partial_x) = L_1(x,\partial_x)^{-1}L_2(x,\partial_x)$.  Let $D_j(x,\partial_x) = Q(x,\partial_x)^{-1}L_j(x,\partial_x)Q(x,\partial_x)$.
Then in particular $W\cdot D_j(x,\partial_x) \subseteq W$ for $j=1,2$.

The subspace $W' = W\cdot D_1 \subseteq W = \Gamma(X\diff\{p\},\mathcal L)$ has finite codimension and $W'\cdot D_1^{-1}= W$ and therefore $W'\cdot D = W'\cdot D_1^{-1}D_2 = W\cdot D_2\subseteq W$.  Let $\mathcal F$ be the submodule of $\mathcal J^{\infty,0}(\mathcal L)$ on $X\diff\{p\}$ generated by $\{j^\infty(v) | v\in W'\}$.
Then $\chi = \chi_D$ extends to a section of the dual of $\mathcal F$ over $X\diff\{p\}$ and thus defines a rational section of the dual of $\mathcal J^{\infty,0}(\mathcal L)$ over $U\subseteq X$.

Conversely, suppose that $\chi$ is a rational section of the dual with $\chi(j^\infty v) = \chi_D(j^\infty v)$ for all $v\in W'$ for some subspace $W'\subseteq W$ of finite codimension.
Then $W'$ has finite codimension in $W$ and we have $W'\cdot D \subseteq W$.
Consequently,
$$\dim(\mathbb L_+ + \mathbb L_+\cdot L)/\mathbb L_+ < \infty.$$
Thus $L(x,\partial_x)$ is a fractional differential operator.
\end{proof}

\subsection{Fractional Krichever correspondence}

The characterization of fractional differential operators in terms of rational sections of the dual of the semi-infinite jet bundle allows us to extend Krichever correspondence to the case of fractional differential operators.
To begin, recall that a rank of an algebra $\mathscr A$ of differential operators is the greatest common divisor of the orders of operators in $\mathcal A$.
In particular, a rank $1$ algebra of commuting differential operators is one containing two operators of relatively prime order.
However, this is \emph{not} the appropriate definition of rank for algebras of fractional differential operators.
For example, if we take any two differential operators $L_1(x,\partial_x)$ and $L_2(x,\partial_x)$ of orders $n$ and $n+1$, respectively, then $L(x,\partial_x) = L_1(x,\partial_x)^{-1}L_2(x,\partial_x)$ is a fractional differential operator of order $1$.
Hence the algebra $\mathbb C[L(x,\partial_x)]$ has what we would call rank $1$ in the differential operator case, even though it really only carries the information of the single operator $L(x,\partial_x)$.

The right notion of the rank of a commutative algebra of fractional differential operators is obtained from understanding the geometric nature of rank in the differential operator setting.
Geometrically, the rank of an algebra of commuting differential operators should be the rank of the corresponding vector bundle over the spectral curve.
In other words, if $\mathscr A$ is a commutative algebra of differential operators and $Q(x,\partial_x) = 1+\sum_{k=1}^\infty b_k(x)\partial_x^{-k}$ is a pseudodifferential operator with $A = Q(x,\partial_x)^{-1}\mathscr A Q(x,\partial_x)\subseteq\mathbb L$, then the rank of $\mathscr A$ is precisely the rank of the torsion-free $A$-module $W = \mathbb L_+\cdot Q(x,\partial_x)$.
Then since maximal commutative algebras $\mathscr K_W$ of fractional differential operators are fields, we can interpret the dimension of the $\mathscr K_W$-span of $W$ as the generic rank of a vector bundle.

Thus we generalize this to the fractional case in the following way.
Any algebra of commuting fractional differential operators is a subalgebra of a \emph{maximal} algebra of commuting fractional differential operators, which are characterized by Theorem \ref{KW and fractional operators} to all be of the form
$${\mathscr{K}}_W=\{U(x,\partial_x)f(\partial_x)U(x,\partial_x)^{-1}: f(z)\in K_W\},$$
for some point $W\in\Gr_+(0)$ with $W = \mathbb L_+\cdot U(x,\partial_x)$ for a pseudodifferential operator $U(x,\partial_x) = 1 + \sum_{n=1}^\infty u_n(x)\partial_x^{-n}$.
Here $K_W$ is the subfield of $\mathbb L$ that we associated to a point $W\in\Gr_+(0)$ in Section \ref{4.1} by
$$K_W = \{f(z)\in\mathbb L: \dim (W + f(z)W)/W < \infty\}.$$
\begin{defn}\label{rank of point}
We define the \vocab{rank of a point $W\in\Gr$} to be the dimension of the $K_W$-subspace of $\mathbb L$ spanned by $W$.
\end{defn}

Note that by Proposition \ref{AW vs KW}, when $A_W$ is nontrivial the field $K_W$ is the fraction field of $A_W$.
Hence in this case the rank of $A_W$ is precisely the rank of $W$.
Furthermore, as the next lemma shows, the finite rank points $W$ are precisely those defining maximal commutative algebras of fractional differential operators.
\begin{lem}
Let $W\in \Gr$.  Then the rank of $W$ is finite if and only if $\mathbb C\subsetneq K_W$.
\end{lem}
\begin{proof}
Note $\mathbb C\subseteq K_W$ for all $W\in\Gr$ and if $\mathbb C=K_W$, then $W$ has infinite rank.

Now suppose that $\mathbb C\neq K_W$.
Then there exists an element $f(z)\in K_W$ of positive degree $r$.
Since $(W + f(z)W)/W$ is finite dimensional, we may choose $\mathbb C$-subspaces $W',E\subseteq W$ with $W'\cap E = \{0\}$ and $W'+E = W$ such that $f(z)W'\subseteq W$.
Now let $\wt E,\wt W\subseteq \mathbb L$ be the $\mathbb C(f(z))$-linear subspace of $\mathbb L$ spanned by $E$ and $W$, respectively.
Clearly $\wt E$ is finite dimensional as a $\mathbb C(f(z))$-vector space.
Also since $W\in\Gr$, we may choose a basis $\{v_j(z): j\geq 0,\ j\notin J\}\subseteq W$ with $\deg(v_j)=j$ for some finite subset $J\subseteq\mathbb Z_+$.
Let $m$ be the largest integer in $J$.
By comparing degrees, we see that $v_k(z)\in\vspan_\mathbb C\{f(z)v_j(z): 0\leq j\leq k-r,\ j\notin J\}$ modulo $\wt E$ for all $k\geq r+m+1$.
Hence the quotient space $\wt W/\wt E$ is spanned over $\mathbb C(f(z))$ by $\{v_j(z): 0\leq j\leq r+m+1,\ j\notin J\}$, and in particular $\wt W/\wt E$ is finite dimensional as a $\mathbb C(f(z))$-vector space.
It follows that $\wt W$ is finite dimensional as a $\mathbb C(f(z))$-vector space.
Since $\mathbb C(f(z))$ is a subfield of $\mathbb L$, a basis for $\wt W$ as a $\mathbb C(f(z))$-vector space will also span the $K_W$-subspace of $\mathbb L$ spanned by $W$.
Hence $W$ has finite rank.
\end{proof}

We can also show that our notion of rank is compatible with the action of fractional differential operators on $\Gr_+(0)$.
\begin{lem}\label{same rank lemma}
Let $Q(x,\partial_x) = 1 + \sum_{k=1}^\infty a_k(x)\partial_x^{-k}$ be a pseudo-differential operator and $L(x,\partial_x)$ be a fractional differential operator.
Then $W = \mathbb L_+\cdot LQ$ and $\wt W = \mathbb L_+\cdot Q$ have the same rank.
\end{lem}
\begin{proof}
First note that since $L(x,\partial_x)$ is fractional,
\begin{align*}
K_W
  & = \{f(z)\in\mathbb C((z^{-1})): LQf(\partial_x)Q^{-1}L^{-1}\ \text{is a fractional diff. oper.}\}\\
  & = \{f(z)\in\mathbb C((z^{-1})): Qf(\partial_x)Q^{-1}\ \text{is a fractional diff. oper.}\} = K_{\wt W}.
\end{align*}
Hence $\mathscr K_W = Q^{-1}L^{-1}\mathscr K_{\wt W}LQ$.
If $\mathcal B = \{v_1(z),\dots,v_r(z)\}\subseteq \wt W$ is a $K_{\wt W}$-basis for $K_{\wt W}\wt W$,
then the $K_W$-span of $W$ is
\begin{align*}
K_WW
  & = W\cdot Q^{-1}\mathscr K_WQ = W\cdot Q^{-1}L^{-1}\mathscr K_{\wt W} LQ\\
  & = \wt WQ^{-1}\mathscr K_{\wt W}LQ = (K_{\wt W}\wt W)\cdot Q^{-1}LQ = (K_{\wt W}\mathcal B)\cdot Q^{-1}LQ\\
  & = \mathcal B\cdot Q^{-1}\mathscr K_{\wt W}LQ = (\mathcal B\cdot Q^{-1}LQ)Q^{-1}L^{-1}\mathscr K_{\wt W}LQ\\
  & = (\mathcal B\cdot Q^{-1}LQ)\cdot Q^{-1}\mathscr K_WQ = K_W(\mathcal B\cdot Q^{-1}LQ).
\end{align*}
Hence $K_W$ is spanned by $\mathcal B\cdot Q^{-1}LQ$, so that the rank of $W$ is less than or equal to the rank of $\wt W$.
Now applying the same argument with $L$ replaced by $L^{-1}$, we get the reverse inequality.  Hence both ranks are the same.
\end{proof}

With our notion of rank for fractional differential operators firmly in place, we are able to state Krichever correspondence in the setting of fractional differential operators.

\begin{thm}
\label{main-k-thm}
Consider a hextuples $(X,p,\mathcal L, z^{-1},\varphi,\chi)$ where $X$ is an algebraic curve, $p\in X$ is a smooth point, $\mathcal L$ is a line bundle over $X$ with trivial cohomology, $z^{-1}$ is a local coordinate of $X$ in an analytic neighborhood $U$ of $p$, $\varphi$ is a local trivialization of $\mathcal L$ over $U$, and $\chi$ is a rational section of the dual of the semi-infinite jet bundle over $X\diff\{p\}$.
Then
$$W = \{\varphi(\chi(j^\infty(s))): s\in \Gamma(Z_p,\mathcal L)\}$$
defines a rank $1$ element of $\Gr_+(0)$.
Furthermore every rank $1$ element of $\Gr_+(0)$ arises in this fashion.
\end{thm}
\begin{remk}
Note that by Theorem \ref{KW and fractional operators}, the point $W\in \Gr_+(0)$ constructed above corresponds to the maximal algebra $\mathscr K_W$ of commuting fractional differential operators, 
given by \eqref{scrKW}. This gives an algebro-geometric construction of all rank 1 maximal algebras of commuting fractional differential operators.
\end{remk}
\begin{proof}
Let us first take a hextuple $(X,p,\mathcal L, z^{-1},\varphi,\chi)$ and define $W$ as in the statement of the theorem.
Via usual Krichever correspondence, the subspace $\wt W = \{\varphi(s): s\in\Gamma(X\diff\{p\},\mathcal L)\}$ defines an element of $\Gr_+(0)$, and we may choose $Q(x,\partial_x) = 1 + \sum_{k=1}^\infty b_k(x)\partial_x^{-k}$ with $\wt W = \mathbb L_+\cdot Q(x,\partial_x)$.
Moreover, $\wt W$ is torsion-free of rank $1$ over $A_W = \Gamma(Z_p,\mathcal L)$, so $\wt W$ is a rank $1$ point of $\Gr_+(0)$.
In particular the $K_{\wt W}$-span $K_{\wt W}\wt W$ of $\wt W$ in $\mathbb L$ is of the form $K_{\wt W}v(z)$ for some $v(z)\in \wt W$.
Furthermore, there exists a pseudodifferential operator $D(x,\partial_x)$ with $\chi = \chi_D$, so that $W= \wt W\cdot D$.
Since $\chi$ is a rational section, the operator $L(x,\partial_x) = Q(x,\partial_x)D(x,\partial_x)Q(x,\partial_x)^{-1}$ is a fractional differential operator.
By Lemma \ref{same rank lemma}, $W$ is rank $1$.

Conversely, suppose that $W\in\Gr_+(0)$ is a rank $1$.
Then in particular $K_W$ is a field extension of $\mathbb C$ of transcendence degree $1$.
Choose a smooth curve $X$ with fraction field $K(X)\cong K_W$ and fix a point $p\in X$.
Choose a uniformizer $z^{-1}$ in $\mathcal O_{X,p}$ so that the image of $K(X)$ in the formal annulus $(\wh{\mathcal O_{X,p}})_{\wh m_p} = \mathbb C((z^{-1}))$ is $K(W)$.
Let $A$ be the image of $\Gamma(X\diff \{p\},\mathcal O_X)$ in $(\wh{ \mathcal O_{X,p}})_{\wh m_p}$ and choose $f(z)\in\mathbb A\diff\mathbb C$ monic of degree $\ell>0$.
Choose a basis $\{v_n(z)\}_{n=0}^\infty\subseteq W$ such that $v_n(z)$ is monic of degree $n$ for all $n$.

Since $f(z)\in K_W$, we know that $\dim((f(z)W+W)/W)<\infty$ and therefore we can choose a finite dimensional subspace $E = \vspan_{\bbc}\{e_1(z),\dots,e_r(z)\}\subseteq \mathbb L$ such that $\deg e_j(z)\leq 0$ for all $j$ and $f(z)W\subseteq W \oplus E$.
By comparing degrees, we see that for all $n$
$$f(z)v_n(z) = v_{n+l}(z) + \sum_{j=1}^{n+\ell-1}c_{nj}v_j(z) + \sum_{k=1}^r b_{nk}e_k(z).$$
Consequently $W$ is contained in the $\mathbb C[f(z)]$-span of $\{v_0(z),\dots,v_{\ell-1}(z)\}\cup\{e_1(z),\dots,e_r(z)\}$.
Now since $W$ is rank $1$, $v_0(z),\dots, v_{\ell-1}(z),e_1(z),\dots, e_k(z)\in K_Wv_0(z) = K(A)v_0(z)$ for $K(A)$ the fraction field of $A$.
Therefore we can choose a polynomial $q(z)\in\mathbb C[z]$ with $q(f(z))v_j(z)/v_0(z)\in A$ and $q(f(z))e_k(z)\in A$ for all $0\leq j\leq \ell-1$ and $0\leq k\leq r$.
This implies that $q(f(z))W\subseteq Av_0(z)$ and therefore $(W+Av_0(z))/Av_0(z)$ is finite-dimensional.

Now let $\mathcal L$ be a line bundle on $X$ with trivial cohomology and choose a local trivialization $\phi$ of $\mathcal L$ near $p$ so that $\wt W := \varphi(\Gamma(Z_p,\mathcal L))$ contains $v_0(z)$.
Then $(\wt W+Av_0(z))/Av_0(z)$ is finite dimensional, so $(\wt W+W)/\wt W$ is finite-dimensional.
Choose pseudodifferential operators $D(x,\partial_x)$ and $Q(x,\partial_x)$ satisfying $\wt W=\mathbb L_+\cdot Q$ and $\wt W\cdot D(x,\partial_x) = W$, with $Q(x,\partial_x)$ monic of order $0$.
Then $(WQ^{-1} + \mathbb L_+)/\mathbb L_+$ is finite-dimensional.
The operator $L(x,\partial_x)) := Q(x,\partial_x)D(x,\partial_x)Q(x,\partial_x)^{-1}$ satisfies $WQ^{-1} = \mathbb L_+\cdot L$, so
$(\mathbb L_+\cdot L + \mathbb L_+)/\mathbb L_+$ is finite dimensional.  Hence $L(x,\partial_x)$ is a fractional differential operator and $\chi=\chi_D$ defines a rational section of the semi-infinte jet bundle on $\mathcal L$.
Thus $W$ is the rank $1$ element of $\Gr_+(0)$ associated with the hextuple $(X,p,\mathcal L,z^{-1},\varphi,\chi)$.
\end{proof}

As a consequence of our characterization, we prove that rank $1$ commutative algebras of fractional differential operators are, up to conjugation, 
fractions of commuting differential operators.
\begin{cor}\label{Cor-classification}
If $\mathscr A$ is a subalgebra of $\mathscr K_W$ for $W$ a rank $1$ point in $\Gr_+(0)$, then there exists a commutative algebra
of differential operators $\mathscr A_0$ and a fractional differential operator $D(x,\partial_x)$ such that
$$D(x,\partial_x)\mathscr A D(x,\partial_x)^{-1} \subseteq \{L_1(x,\partial_x)^{-1}L_2(x,\partial_x): L_j(x,\partial_x)\in\mathscr A_0,\ j=1,2\}.$$
\end{cor}
\begin{proof}
Let $W$ be a rank $1$ point in $\Gr_+(0)$ corresponding to $(X,p,\mathcal L, z^{-1},\varphi,\chi)$ as in the previous theorem, and let $Q(x,\partial_x)$ be the wave operator satisfying $W_0 = \mathbb C[z]\cdot Q(x,\partial_x)$ for $W_0$ the point in $\Gr_+(0)$ associated with $(X,p,\mathcal L,z^{-1},\varphi)$.
Note that $W_0$ also has rank $1$ and $A_{W_0}\neq\mathbb C$, so by \ref{AW vs KW} $\mathscr K_{W_0}$ is the fraction field of $\mathscr A_{W_0}$.

Express $\chi$ in a neighborhood of $p$ as $\chi = \chi_B$ for some pseudodifferential operator $B(x,\partial_x)$.
Since $\chi$ is a fractional section of the dual, we know that $D(x,\partial_x) = Q(x,\partial_x)B(x,\partial_x)Q(x,\partial_x)^{-1}$ is a fractional differential operator by Theorem  \ref{cool catz}.  Moreover,
$$W = \chi_D(W_0) = W_0\cdot D(x,\partial_x) = \mathbb C[z]\cdot Q(x,\partial_x)B(x,\partial_x) = \mathbb C[z]\cdot D(x,\partial_x)Q(x,\partial_x).$$
It follows that if $R(x,\partial_x)\in \mathscr K_W$, then $D(x,\partial_x)R(x,\partial_x)D(x,\partial_x)^{-1}\in \mathscr K_{W_0}$.
\end{proof}

\section{Examples}
In this section we illustrate the results with several explicit examples. The first example provides a bridge to the classical Krichever correspondence for nonsingular curves 
where all quantities can be written explicitly in terms of $\theta$-functions, and the fractional operators can be naturally related to generalized eigenvalue problems.
Some key features of this example are also described in an important sub-example.
We also explicitly describe how these two examples are obtained from Theorem \ref{main-k-thm}. 
We then present two examples that illustrate Theorem \ref{main-k-thm}, and in particular, the use of sections of dual semi-infinite jet bundles in it. 

\subsection{Krichever's construction for nonsingular curves}
Firstly, we recall the explicit form of Krichever's correspondence from Theorem \ref{ktheorem} in the case of non-singular curves. 
Let $X$ be a non-singular complex algebraic curve of genus $g$ and let $(X,{\infty},z,D)$ be the data in Section~\ref{2.2}. We fix a 
canonical basis $\{\alpha_j,\beta_j\}_{j=1,\dots,g}$ for $H_1(X,\bbz)$, i.e. $\alpha_j\circ\beta_k=\delta_{jk}$. 
Let $\{\omega_j\}_{j=1,\dots,g}$ be a basis of the space of holomorphic $1$-forms normalized by the condition 
\begin{equation*}
\oint_{\alpha_j}\omega_k=\delta_{jk},
\end{equation*}
and denote by $B$ the matrix of $\beta$-periods
\begin{equation*}
B_{jk} = \oint_{\beta_j}\omega_k.
\end{equation*}
Let ${\mathfrak{J}}(X)=\bbc^g/\{\bbz^g+B\bbz^g\}$ denote the Jacobian of $X$, and $A:X\rightarrow {\mathfrak{J}}(X)$ be the Abel map
\begin{equation}						\label{D.1}
A(P)=\left(\int_{P_0}^P\omega_1,\int_{P_0}^P\omega_2,\dotsc,
		\int_{P_0}^P\omega_g\right),
\end{equation}
where $P_0$ is a fixed point on the Riemann surface, and the path of 
integration is the same in all integrals. 
The {\bf Riemann theta function} is defined by
\begin{equation}\label{RTF}
\theta(Z|B)=\sum_{N\in\bbz^g}\exp\left(\pi i\langle BN,N\rangle+
2\pi i\langle N,Z\rangle\right), \text{ where }Z\in\bbc^g \text{ and }\langle N,Z\rangle=\sum_{j=1}^gN_jZ_j.
\end{equation}
Let $\eta^{(n)}$ denote the normalized Abel differential of 
second kind with a pole of order $(n+1)$ of the form 
\begin{equation*}
\eta^{(n)}=dz^n+\text{(holomorphic part)},
\end{equation*}
with vector of $\beta$-periods $(2\pi i)U_n$:
\begin{equation*}
(U_n)_j=\frac{1}{2\pi i}\oint_{\beta_j}\eta^{(n)}.
\end{equation*}
For $P$ close to ${\infty}$ we have
\begin{equation}						\label{D.3}
\int_{P_0}^P\eta^{(n)}=z^n+c_{n0}+\sum_{j=1}^{\infty}\frac{c_{nj}}{j}z^{-j}.
\end{equation}
If $\cK=(\cK_1,\dotsc,\cK_g)$ is the vector of Riemann constants, then the Baker-Akhiezer function corresponding to the data $(X,{\infty},z,D)$ is
\begin{equation}						\label{D.4}
\begin{split}
\psi(\vec t, P)=&\exp\left(\sum_{n=1}^{\infty}t_n\left(\int_{P_0}^P\eta^{(n)}-c_{n0}
\right)\right)\\
&\times\frac{\theta(A(P)+\sum t_nU_n-A^{(g)}(D)-\cK)}{\theta(A({\infty})+
\sum t_nU_n-A^{(g)}(D)-\cK)}
\frac{\theta(A({\infty})-A^{(g)}(D)-\cK)}{\theta(A(P)-A^{(g)}(D)-\cK)},
\end{split}
\end{equation}
see \cite{Krichever1}. In a neighborhood of $\infty$, we can replace $P$ by $z$, and the stationary Baker-Akhiezer function $\psi(x, z)$ can be obtained by setting $t_1=x$ and $t_j=0$ for $j>1$. The $\tau$-function for the corresponding point $W\in\Gr$ is given by the formula
\begin{equation}						\label{D.12}
\tau(t)=\exp\left(-\frac{1}{2}\sum_{j,k\geq1}c_{jk}t_jt_k+\sum_{j\geq1}
\lambda_jt_j\right)\theta\left(A(\infty)+\sum_{j\geq1}t_jU_j-A^{(g)}(D)-\cK\right),
\end{equation}
for appropriate normalizing constants $\lambda_j$, see \cite[Section 9]{Segal-Wilson}.

\subsection{Fractional operators associated with hyperelliptic curves}

Consider the affine curve
\begin{equation}\label{HEC}
w^2=\prod_{j=1}^{2g+2}(z-z_j),
\end{equation}
where $z_1,\dots,z_{2g+2}$ are distinct complex numbers. Let $X$ denote the completion obtained by adding two points $\{\infty,-\infty\}$. Using $z^{-1}$ as a local parameter at these points, we have 
\begin{align*}
& w(P)=z^{g+1}\sqrt{\prod_{j=1}^{2g+2}(1-z_j/z)} && \text{in a neighborhood of $\infty$} \\
& w(P)=-z^{g+1}\sqrt{\prod_{j=1}^{2g+2}(1-z_j/z)} && \text{in a neighborhood of  $-\infty$,} 
\end{align*}
where the square root is fixed so that $\sqrt{1}=1$.  The basis of holomorphic 1-forms can be computed explicitly by normalizing the basis $\{z^{k-1}dz/w\}_{k=1,\dots,g}$.

Applying Krichever's construction in the previous subsection with $P_0=-\infty$, we can construct monic differential operators $L_1=L_1(x,\partial_x)$ and $L_2=L_2(x,\partial_x)$ of orders $1$ and $2$, respectively, satisfying the generalized eigenvalue problem
\begin{equation}						\label{K.3}
L_2 \psi(x, P)=z(P)L_1 \psi(x, P).
\end{equation}
Indeed,  if we set 
$$L_1=\partial_x+b_0, \text{ where }b_0=-\partial_x \log (\psi(x, -\infty))$$ 
then  $L_1 \psi(x, P)$ cancels the pole of $z$ at $-\infty$ on the right-hand side of equation \eqref{K.3}. In a neighborhood of $\infty$, the right-hand side will have the expansion 
$$\left(z^2+\psi_1(x)z+\psi_0(x)+O(z^{-1})\right)e^{xz}.$$
Therefore, we can construct a monic differential operator $L_2=\partial_x^2+a_1\partial_x+a_0$ of order $2$ such that 
$L_2 \psi(x, z)-zL_1 \psi(x, z) =O(z^{-1})e^{xz}$ and the uniqueness of the Baker-Akhiezer function shows that this term must vanish, i.e. \eqref{K.3} holds.

If we use the KP flows, we can write explicit formulas for the coefficients of the operators $L_j$ in terms of the $\tau$-function in \eqref{D.12}. Indeed, expanding \eqref{K.3} in a neighborhood of $\infty$ and comparing the coefficients of $z$, after we canceling the exponent, we see that $a_1=b_0$ and therefore
\begin{equation*}
a_1=b_0=-\partial_x \log (\psi(\vec t, -\infty)).
\end{equation*}

The coefficients of $z^0$ now show that $a_0=-\partial_xu_1$. On the other hand, $u_1=-\partial_x \log \tau$ and therefore
\begin{equation}\label{e.3}
a_0=\partial_x^2 \log \tau.
\end{equation}
If we compare the coefficients of $z^{-1}$ in \eqref{K.3}  we see that
$$a_0u_1 + a_1u_1' + u_2' + u_1''=0,$$
which combined with Sato's formulas
$$u_1=-\partial_x \log \tau, \qquad u_2=\frac{(\partial_x^2-\partial_2)\tau}{2\tau},$$
yields a formula for $a_1$ in terms of the $\tau$-function:
\begin{equation}\label{e.4}
a_1=-\frac{(\partial_x^3+\partial_x\partial_2)\log\tau}{2\partial_x^2 \log \tau}.
\end{equation}

In a neighborhood of $\infty$, equation~\eqref{K.3} can be rewritten as
\begin{equation}						\label{D.21}
L\psi(x, z)=z\psi(x, z), \qquad \text{where }L=L_1^{-1}L_2=\partial_x+(\partial_x+a_1)^{-1}a_0,
\end{equation}
and the fractional operator $L$ belongs to the rank 1 point $W\in\Gr$ corresponding to the Krichever's data $(X,{\infty},z,D)$. Note that the operators $L_1$ and $L_2$ in \eqref{D.21} do not commute, but we can represent $L$ also
as $L=\tilde{L}_1^{-1}\tilde{L}_2,$
where $\tilde{L}_1$ and $\tilde{L}_2$ are commuting operators from the algebra \eqref{Schur-to-CAode} as stated in Corollary~\ref{Cor-classification}. This is equivalent to showing that $z$ belongs to the fraction field of $A_W$, which was proved in Proposition~\ref{AW vs KW}. For the hyperelliptic curve \eqref{HEC}, we can easily construct explicitly two functions from $A_W$ whose ratio is equal to $z$. Indeed, since $z$ has a simple pole at $-\infty$, we can choose constants $c_0,c_1,\dots,c_{g}$ such that the function
\begin{equation}						\label{D.22}
\tilde{w}=\frac{1}{2}\left(w+z^{g+1}+\sum_{j=0}^{g}c_jz^j\right)
\end{equation}
is holomorphic on $X\setminus\{\infty\}$ and vanishes at $-\infty$. Then $z=(z\tilde{w})/\tilde{w}$, where $\tilde{w},z\tilde{w}\in A_W$.
Note that 
$$\tilde{w}(P)=z^{g+1}+O\left(z^{g}\right) \qquad  \text{in a neighborhood of $\infty$},$$
and  therefore, we can construct monic differential operators $\tilde{L}_1$ and  $\tilde{L}_2$ of orders $g+1$ and $g+2$, respectively, such that
\begin{equation*}
\tilde{L}_1\psi(x, z)=\tilde{w}\,\psi(x, z), \quad \text{ and }\quad \tilde{L}_2\psi(x, z)=z\tilde{w}\,\psi(x, z).
\end{equation*}
Thus, for the operator $L$ in \eqref{D.21} we have 
$L=L_1^{-1}L_2=\tilde{L}_1^{-1}\tilde{L}_2,$
where $\tilde{L}_1$ and $\tilde{L}_2$ are commuting operators of orders $g+1$ and $g+2$, respectively. 
\begin{remk}
When $g=1$, \eqref{HEC} is replaced by the elliptic curve 
\begin{equation}\label{EC}
w^2=p_4(z)=(z-z_1)(z-z_2)(z-z_3)(z-z_4),
\end{equation}
and the formulas above can be expressed in terms of well-known elliptic functions. Indeed, we can take
$$\omega_1=\frac{1}{a}\frac{dz}{\sqrt{p_4(z)}}, \text{ where }a=2\int_{z_1}^{z_2}\frac{dz}{\sqrt{p_4(z)}} \quad
\text{ and }\quad B=\frac{\int_{z_2}^{z_3}\frac{dz}{\sqrt{p_4(z)}}}{\int_{z_1}^{z_2}\frac{dz}{\sqrt{p_4(z)}}}.$$
The Riemann theta function \eqref{RTF} reduces to the  {\bf Jacobi theta function} 
\begin{equation*} 
\theta(z;\tau)=\sum_{n=-\infty}^{\infty}e^{\pi in^2\tau +2\pi inz},\qquad \text{ where }\tau=B.
\end{equation*}
Using the formula
$$\wp(z,\tau)=-\partial_z^2 \log\theta\left(z+\frac{\tau}{2}+\frac{1}{2};\tau\right)+\mathrm{constant},$$
we can express the coefficients of the operators $L_1$ and $L_2$ given in \eqref{e.3}-\eqref{e.4} in terms of the {\bf Weierstrass $\wp$-function}.
This gives that in this sub-example, the coefficients $c_0$ and $c_1$ in \eqref{D.22} are explicitly given by
$$c_1=-\frac{1}{2}\sum_{j=1}^{4}z_j \qquad\text{ and }\qquad c_0 =\frac{1}{4}\sum_{1\leq i<j\leq 4}z_iz_j-\frac{1}{8}\sum_{i=1}^{4}z_i^2.$$
\end{remk}

\subsection{An Example of Krichever Correspondence}
In this subsection, we will formally construct a point in $\Gr_+(0)$ from a Krichever qunituple of data, corresponding to a maximal algebra of commuting fractional differential operators. To start, we will construct a Krichever quintuple associated to a line bundle over an elliptic curve.

Consider the elliptic curve
$$X = \{[X_0:X_1:X_2]\in \mathbb P^2_{\mathbb C}: X_1^2X_2 = 4X_0^3 - g_2X_0X_2^2-g_3X_2^3\}$$
along with the distinguished point $\infty = [0:1:0]$, where here $g_2,g_3\in\mathbb C$ are chosen so that $X$ is nonsingular.
The line bundles with trivial cohomology on $X$ are determined by a choice of closed point $q\in X\diff\{\infty\}$.
Specifically, given $q = [a:b:1] \in X\diff\{\infty\}$, we have the line bundle $\mathcal O(D)$ associated to the Weyl divisor $D = [\infty]-[q]$ 
Let $U=\{X_1\neq 0\}\subseteq X$.
The sections of the associated line bundle $\mathcal O(D)$ over $U$ are
$$\Gamma(U,\mathcal O(D)) = \{f\in k(X): f \equiv 0\ \text{or}\ (\text{div}(f)+D)|_U\geq 0\}.$$
The sections of the Serre twist $\mathcal O(-1)$ over $U$ are
$$\Gamma(U,\mathcal O(-1)) = \{f(X_0/X_1,X_2/X_1)/X_1: f\ \text{homog. poly.}\}.$$
The homogeneous rational function $X_1/X_0$ has a zero of order $1$ at $\infty$.
Consequently for any neighborhood $U$ of $\infty$, the module homomorphism 
$$\Gamma(U,\mathcal O(D))\rightarrow \Gamma(U,\mathcal O(-1)),\ \ f\mapsto (X_1/X_0)\frac{f}{X_1}$$
is well-defined and localizes to an isomorphism of stalks $\varphi: \mathcal O(D)_\infty\cong \mathcal O_{X,\infty}$.

The Weierstrass $\wp$-function for this elliptic curve defines a parametrization
$$\phi: \mathbb C \rightarrow X,\ \ w\mapsto \left\lbrace\begin{array}{cc}[\wp(w):\wp'(w):1], & w\notin\Lambda\\ \infty, & w \in \Lambda\end{array}\right.$$
where $\Lambda$ is the lattice of singularities of $\wp(z)$, and restricts to a holomorphism $\mathbb C/\Lambda\rightarrow X$.
Thus $z^{-1} := w$ defines a local uniformizer in a local analytic neighborhood of $\infty\in X$.
Here, we are viewing $\wp(z^{-1})$ as a Laurent series
$$\wp(z^{-1}) = z^2 + \sum_{n=1}^\infty (2n+1)G_{2n+2}z^{-2n},\ \ \text{for}\ G_n = \sum_{0\neq\lambda\in\Lambda} \lambda^{-n},$$
defined in an annular neighborhood of $\infty$ in the complex plane, which by rescaling we can assume contains the unit circle $S^1$.

Fixing the value of $q = [a:b:1]\in X$, we obtain a traditional Krichever quintuple $(X,\infty,\mathcal O(D), z^{-1},\varphi)$ whose associated point in $\Gr_+(0)$ is given by
restricting sections of $\mathcal O(D)$ over $X\diff\{\infty\}$ to the stalk $\mathcal O_{X,\infty}$ via the trivialization $\varphi$, and then precomposing with $\phi$.
Thus we find
\begin{align*}
W & = \{f(\wp(z^{-1}),\wp'(z^{-1}))/\wp(z^{-1}): f(x,y)\in\mathbb C[x,y],\ f(a,b) = 0\},
\end{align*}
which satisfies $A_W = \mathbb C[\wp(z^{-1}),\wp'(z^{-1})]$.

Choose $Q(x,\partial_x) = 1 + \sum_{n=1}^\infty \mu_n(x) \partial_x^{-n}$ with $W = \bbc[z]\cdot Q$.
As in the usual Krichever correspondence, the algebra $\mathbb C[\wp(\partial_x^{-1}),\wp'(\partial_x^{-1})]$ conjugates under $Q$ to a commutative algebra of differential operators.
In particular there are operators $L_2(x,\partial_x)$ and $L_3(x,\partial_x)$ satisfying
\begin{align*}
L_2(x,\partial_x) &= Q(x,\partial_x)\wp(\partial_x^{-1})Q(x,\partial_x)^{-1} = \partial_x^2 + u(x),\\
L_3(x,\partial_x) &= Q(x,\partial_x)\wp'(\partial_x^{-1})Q(x,\partial_x)^{-1} = -2\partial_x^3 + v(x)\partial_x + w(x).
\end{align*}

The value of $u(x)$ is given by the $\tau$-function of $W$ by $u(x) = 2\frac{d^2}{dx^2}\log\tau_W(x)$.
The $\tau$-function in turn can be expressed in terms of the theta function of the elliptic curve, as alluded to in the previous section.
Alternatively, as a more pedestrian approach we can note that the operators $L_2$ and $L_3$ must commute and satisfy the same algebraic relation as $\wp(z^{-1})$ and $\wp'(z^{-1})$, namely
$$L_3(x,\partial_x)^2 = 4L_2(x,\partial_x)^3-g_2L_2(x,\partial_x)-g_3.$$
From this, one may deduce
\begin{align}
L_2(x,\partial_x) &= \partial_x^2 - 2\wp(x-c_q),\\
L_3(x,\partial_x) &= -2\partial_x^3 + 6\wp(x-c_q)\partial_x + 3\wp'(x-c_q).
\end{align}
for some constant $c_q\in\mathbb C$ whose value depends on the choice of $q = [a:b:1]$.
Thus the algebra $A_W$ conjugates to the algebra of commuting differential operators commuting with $L_2(x,\partial_x)$ and $L_3(x,\partial_x)$.
The field associated to the point $W$ is $K_W = \mathbb C(\wp(z^{-1}),\wp'(z^{-1}))$, so the associated maximal algebra of commuting fractional differential operators is
$$\mathscr K_W =  \mathbb C(L_2(x,\partial_x),L_3(x,\partial_x)).$$

\subsection{An example of extended Krichever correspondence}

Next, we extend the example of the previous subsection to an example of Krichever correspondence for a more general algebra of fractional differential operators.

Let $W$ and $Q$ be as in the previous example.
Consider the operator
$$P(z,\partial_z) = -\frac{1}{\wp(z^{-1})}z^2\partial_z\wp(z^{-1}) = -z^2\partial_z + \frac{\wp'(z^{-1})}{\wp(z^{-1})}.$$
It is easy to check that $P(z,\partial_z)\cdot W\subseteq W$ and therefore $U = Q\wt PQ^{-1}$ must be a differential operator for $\wt P(x,\partial_x) = -x\partial_x^2 + \wp'(\partial_x^{-1})/\wp(\partial_x^{-1})$.

Comparing orders and coefficients, we determine $U=-x\partial_x^2 - \partial_x + g(x)$ for some function $g(x)$.
Furthermore, since $[P(z,\partial_z),\wp(z^{-1})] = \wp'(z^{-1})$, we know that $[L_2(x,\partial_x),U(x,\partial_x)] = L_3(x,\partial_x)$, from which one may deduce $g'(x) = 2x\wp'(x-c_q)+3\wp(x-c_q)$.
This shows us that 
$$U(x,\partial_x) = -x\partial_x^2 -\partial_x + 2x\wp(x-c_q) - \zeta(x-c_q),$$
for $\zeta(z)$ the Weierstrass zeta function of the elliptic curve.

Using the notation of the previous subsection, consider the rational section $\chi$ of the dual of the weak jet bundle $J^{\infty,0}(\mathcal L)$, defined over the formal annulus at $p$ by
$$\chi(j^{\infty}(v)) = P(z,\partial_z)\cdot v(z)$$

Alternatively, in terms of the action of pseudo-differential operators on $\mathbb L$, we can write $\chi = \chi_{\wt P}$.
The associated point in $\Gr_+(0)$ is defined by
\begin{align*}
\wt W
 & = \{\varphi\circ \chi(j^{\infty}(f)): f\in \Gamma(X\diff\{p\},\mathcal O(D))\}\\
 & = \left\lbrace P(z,\partial_z)\cdot v(z) : v(z)\in W\right\rbrace = W\cdot \wt P(x,\partial_x)
\end{align*}
where here $W = \bbc[z]\cdot Q(x,\partial_x)$ is the point in Sato's Grassmannian from the previous example.
Furthermore, we can write
$$\wt W = \bbc[z]\cdot Q(x,\partial_x)\wt P(x,\partial_x) = \bbc[z]\cdot U(x,\partial_x)Q(x,\partial_x).$$

In this case $A_W=\mathbb C$, which would be a problem for the classical theory.
However, since $U(x,\partial_x)$ is a differential operator, it follows that
$K_{\wt W}= K_W$, and in particular is nontrivial.
The associated maximal field of commuting fractional differential operators is
$$\mathscr K_{\wt W} = \mathbb C(\wt L_2(x,\partial_x), \wt L_3(x,\partial_x))$$
for the fractional operators
$$\wt L_j(x,\partial_x) = U(x,\partial_x)L_j(x,\partial_x)U(x,\partial_x)^{-1}.$$

\section{Appendix: jet bundles}
Jet bundles are vector bundles which encode the data of a function and its derivatives, thereby allowing for a formal exploration of nonlinear partial differential equations in a purely geometric setting.
A comprehensive introduction to jet bundles in the context of vector bundles on smooth real manifolds can be found in \cite{Saunders}.
For the purposes of this paper, we will focus on the theory of jet bundles from the point of view of algebraic geometry, based on the presentation in \cite{Vakil}.
In the literature, the sheaf of jet bundles is alternatively called the \vocab{sheaf of principal parts} and is defined in full generality in \cite[Section 16]{EGA4}.

Throughout this section, $X$ will denote a complex algebraic variety which is Cohen-Macaulay.

\subsection{Finite jets}

Let $p\in X$ and let $z = (z_1,\dots, z_n)$ be local coordinates for $X$ in a neighborhood $U$ of $p$.
Intuitively, the bundle $\mathcal J^m(\mathcal O_X)$ of $m$-jets of the structure sheaf $\mathcal O_X$ is a vector bundle on $X$ whose sections are generated as an $\mathcal O_X$-module by the $m$-truncated Taylor series of functions in $\mathcal O_X$.
Adopting multi-index notation, the Taylor series of a function $f: X\rightarrow \mathbb C$ at a point $p = (p_1,\dots,p_n)\in X$ truncated to degree $m$,
$$\tau_{f,m}(z,p) = \sum_{|I|\leq m} \frac{(z-p)^I}{I!}\frac{\partial^I f}{\partial z^I}|_p,$$
may be viewed as a function $\tau_{f,m}$ on $X\times X$, ie. an element of $\mathcal O_{X\times X}$.
Thus we can naturally view $\mathcal J^m(\mathcal O_X)$ as a subsheaf of $\mathcal O_{X\times X}$, whose cokernel is $\mathcal I_\Delta^{m+1}$ where here $\mathcal I_{\Delta}$ is the diagonal ideal of $X\times X$, defined by
$$\mathcal I_{\Delta}(U\times V) = \{g(x,y)\in\mathcal O_{X,X}(U\times V): g(x,x) = 0,\ \forall x\in U\cap V\}.$$
This motivates the definition of the jet bundle $\mathcal J^m(\mathcal O_X)$ as $\mathcal O_{X\times X}/\mathcal I_\Delta^{m+1}$.
More generally, we have the following definition.

\begin{defn}
Let $\mathcal E$ be a locally free $\mathcal O_X$-module of finite rank.
The \vocab{bundle of $m$-jets of $\mathcal E$} is defined to be
\begin{equation}
\mathcal J^m(\mathcal E) = q_{1*}\left(\mathcal O_{X\times X}/\mathcal I_\Delta^{m+1}\otimes q_2^*\mathcal E\right)
\end{equation}
where here $\mathcal I_\Delta$ is the ideal sheaf of the image of the diagonal embedding $\Delta: X\rightarrow X\times X$ and $q_j: X\times X\rightarrow X$ are the projection on the first and second components (respectively) for $j=1,2$.
\end{defn}
Over an affine neighborhood $U=\spec(A)\subseteq X$, we may identify $\mathcal J^m(\mathcal E)$ with the $A$-module
\begin{equation}\label{local module structure}
\Gamma(U,\mathcal J^n(\mathcal E)) = (A\otimes_{\bbc}\Gamma(U,\mathcal E))/I_\Delta^{m+1}(A\otimes_{\bbc}\Gamma(U,\mathcal E))
\end{equation}
where the action by $A$ occurs on the first entry of the tensor product and $I_\Delta = \langle\{a\otimes 1-1\otimes a: a\in A\}\rangle$.
Note in particular that $\mathcal J^0(\mathcal E)=\mathcal E$ and that $\mathcal J^1(\mathcal E)$ is related to the cotangent bundle on $X$ by $\mathcal J^1(\mathcal E)\cong (\mathcal O_X\otimes \Omega_{X/\mathbb C})\otimes \mathcal E$.

Each section $s\in \Gamma(U,\mathcal E)$ over the affine open $U$ induces a section $j^ms\in\Gamma(U,\mathcal J^m(\mathcal E))$ defined by $1\otimes q_2^*s$.
\begin{defn}
Let $s\in \Gamma(U,\mathcal E)$ be a section.
The \vocab{$m$-jet of $s$} is the section of $\mathcal J^m(\mathcal E)$ over $U$ defined by
\begin{equation}\label{j_k}
j^m(s) = 1\otimes q_2^*s\in \mathcal O_{X\times X}/\mathcal I_\Delta^{m+1}\otimes q_2^*\mathcal E.
\end{equation}
\end{defn}

For integers $m>\ell\geq 0$, there is a natural module epimorphism
\begin{equation}
\pi_{m,\ell}: \mathcal J^m(\mathcal E)\rightarrow \mathcal J^\ell(\mathcal E)
\end{equation}
induced by the surjection $\mathcal O_{X\times X}/\mathcal I_\Delta^{m+1}\rightarrow \mathcal O_{X\times X}/\mathcal I_\Delta^{\ell+1}$.
These projections correspond to truncating Taylor polynomials.
\begin{prop}
Let $\mathcal E$ be a locally free $\mathcal O_X$ module of finite rank over a scheme $X$.  Then there is a short exact sequence
$$0\rightarrow q_{1*}(\mathcal I_\Delta^m/\mathcal I_\Delta^{m+1})\otimes\mathcal E\rightarrow\mathcal J^m(\mathcal E)\rightarrow\mathcal J^{m-1}(\mathcal E)\rightarrow 0.$$
\end{prop}
\begin{proof}
Since $\mathcal E$ is locally free, the short exact sequence
$$0\rightarrow \mathcal I_\Delta^m/\mathcal I_\Delta^{m+1}\rightarrow \mathcal O_{X\times X}/\mathcal I_\Delta^{m+1}\rightarrow\mathcal O_{X\times X}/\mathcal I_\Delta^m\rightarrow 0.$$
extends to a short exact sequence
$$0\rightarrow \mathcal I_\Delta^m/\mathcal I_\Delta^{m+1}\otimes q_2^*\mathcal E\rightarrow \mathcal O_{X\times X}/\mathcal I_\Delta^{m+1}\otimes q_2^*\mathcal E\rightarrow\mathcal O_{X\times X}/\mathcal I_\Delta^m\otimes q_2^*\mathcal E\rightarrow 0.$$
Noting that $\mathcal I_\Delta^m/\mathcal I_\Delta^{m+1}\otimes q_2^*\mathcal E\cong \mathcal I_\Delta^m/\mathcal I_\Delta^{m+1}\otimes q_1^*\mathcal E$ and pushing forward by $q_1$ while applying the projection formula \cite[Exercise III.8.3]{Hartshorne}
$$0\rightarrow q_{1*}(\mathcal I_\Delta^m/\mathcal I_\Delta^{m+1})\otimes \mathcal E\rightarrow \mathcal O_{X\times X}/\mathcal I_\Delta^{m+1}\otimes q_2^*\mathcal E\rightarrow\mathcal O_{X\times X}/\mathcal I_\Delta^m\otimes q_2^*\mathcal E\rightarrow R^1q_{1*}(\mathcal I_\Delta^m/\mathcal I_\Delta^{m+1})\otimes \mathcal E.$$
Now since $\mathcal I_\Delta$ is supported on the diagonal where $q_1$ restricts to an isomorphism, we have $R^1q_{1*}(\mathcal I_\Delta^m/\mathcal I_\Delta^{m+1}) = 0$.
\end{proof}

We can realize the above result very naturally in terms of a the sheaf of K\"ahler differentials $\Omega_{X/\mathbb C}$ of $X$.
First, recall that the map $d: \mathcal O_X\rightarrow q_{1*}(\mathcal I_{\Delta}/\mathcal I_\Delta^2)$ defined affine locally by $s\mapsto s\otimes 1-1\otimes s$ is a differental, where here we are using the natural identification $\Gamma(U,\mathcal O_{X\times X})\cong \Gamma(U,\mathcal O_X)\otimes_{\mathbb C}\Gamma(U,\mathcal O_X)$.
Intuitively, it takes a function $f(x)$ defined locally on an open $U\subseteq X$ to the function $f(x)-f(y)$ on $X\times X$ modulo the ideal $I_\Delta^2$.
By the universal property of $\Omega_{X/\mathbb C}$, the differential induces a module homomorphism $\Omega_{X/\mathbb C}\rightarrow q_{1*}(\mathcal I_{\Delta}/\mathcal I_\Delta^2)$, which is actually an isomorphism.
\begin{defn}
The map $d: \mathcal O_X\rightarrow q_{1*}(\mathcal I_\Delta/\mathcal I_{\Delta}^2)$ is called the \vocab{universal derivative}.
\end{defn}

Furthermore, since $X$ is Cohen-Macaulay there exists a natural isomorphism \cite[Theorem 8.21.A]{Hartshorne} $\text{Sym}^m(\mathcal I_\Delta/\mathcal I_\Delta^2)\cong \mathcal I_\Delta^n/\mathcal I_\Delta^{n+1}$, where here $\text{Sym}^m(\mathcal F)$ denotes the symmetric product of an $\mathcal O_X$-module $\mathcal F$.
Hence we have isomorphisms
$$\text{Sym}^m(\Omega_{X/\mathbb C})\cong\text{Sym}^mq_{1*}(\mathcal I_\Delta/\mathcal I_\Delta^2)\cong q_{1*}\text{Sym}^m(\mathcal I_\Delta/\mathcal I_\Delta^2)\cong q_{1*}(\mathcal I_\Delta^m/\mathcal I_\Delta^{m+1}).$$
To summarize, we have the following result.
\begin{prop}
Let $\mathcal E$ be a locally free $\mathcal O_X$ module of finite rank over a Cohen-Macaulay scheme $X$.  Then there is a short exact sequence
$$0\rightarrow\text{Sym}^m(\Omega_{X/\mathbb C})\otimes\mathcal E\rightarrow\mathcal J^m(\mathcal E)\xrightarrow{\pi_{m,m-1}}\mathcal J^{m-1}(\mathcal E)\rightarrow 0$$
where the first arrow is locally defined by the map
$$(df_1\otimes \dots\otimes df_m)\otimes s_j\mapsto (df_1)(df_2)\dots(df_n)(1\otimes s_j),$$
using the local module structure for the jet bundle from \eqref{local module structure}, for $d$ the universal derivative of $X$ and $s_1,\dots, s_r$ a basis of sections for $\mathcal E$ as a $\mathcal O_X$-module.
\end{prop}
\begin{proof}
This follows from the discussion of the previous paragraph, along with retracing the precise values of the isomorphisms.
Note that the product in the right hand side of the last equation is taken in terms of representatives of each equivalence class $df_j$ modulo $\mathcal I_\Delta^2$ and that the final product is independent of the choice of representatives, since it's value is taken modulo $\mathcal I_\Delta^{m+1}$.
\end{proof}

When $X$ is nonsingular, the sheaf $\Omega_{X/\mathbb C}$ is locally free of finite rank, so that the above short exact sequence splits.  In particular, we can then write
\begin{equation}\label{decomposition of jet bundle}
\rho_m: \left(\mathcal O_X\oplus\sum_{j=1}^m\text{Sym}^j(\Omega_{X/\mathbb C})\right)\otimes\mathcal E\cong \mathcal J^m(\mathcal E)
\end{equation}
via the map defined affine locally by
\begin{equation}\label{decomposition of jet bundle map}
\rho_m: \left[f_0+ \sum_{j=1}^m (df_1\otimes\dots\otimes df_j)\right]\otimes s_j\mapsto \left[f_0\otimes 1 + \sum_{j=1}^m \left(\prod_{k=1}^j df_k\right)\right](1\otimes s_j),
\end{equation}
using the local module structure for the jet bundle from \eqref{local module structure}, with $s_1,\dots, s_r$ a basis of sections for $\mathcal E$ as a $\mathcal O_X$-module.
In particular, in this case the jet bundle is also locally free of finite rank.
Since localization is exact, this equality is certainly true over the smooth locus of $X$.
Notice that in this case
$$
\iota_{\ell,m}:\mathcal J^\ell(\mathcal E)\cong \left(\mathcal O_X\oplus\sum_{j=1}^{\ell}\text{Sym}^j(\Omega_{X/\mathbb C})\right)\otimes\mathcal E\xrightarrow{\subseteq}
\left(\mathcal O_X\oplus\sum_{j=1}^m\text{Sym}^j(\Omega_{X/\mathbb C})\right)\otimes\mathcal E\cong\mathcal J^m(\mathcal E)$$
defines a canonical splitting
\begin{center}
\begin{tikzcd}
    0\arrow{r} & \text{Sym}^m(\Omega_{X/\mathbb C})\otimes \mathcal E\arrow{r} & \mathcal J^m(\mathcal E)\arrow{r}{\pi_{m,m-1}} & \mathcal J^{m-1}(\mathcal E)\arrow{r}\arrow[bend left=33]{l}{\iota_{m-1,m}} & 0
\end{tikzcd}
\end{center}

Suppose that $U=\spec(A)$ is an affine open subset of $X$ over which $\mathcal E$ is free with basis $s_1,\dots, s_r$, for $A = \mathbb C[x_1,\dots, x_N]/I$.
We can write the $m$-jet of $s\in\Gamma(U)$ in terms of the decomposition \eqref{decomposition of jet bundle} as follows.
First, we may write $s = \sum_{j=1}^r f_j(x)s_j$ for $f_j(x)=f_j(x_1,\dots,x_N)\in A$.
Then for each $f_j(x)$, we can write
$$f_j(y) = \sum_{|I|\leq \ell_j} \frac{f_i^{(I)}(x)}{I!}(y-x)^I$$
in $\mathbb C[x_1,\dots, x_N,y_1,\dots,y_N]$,
where here $\ell_j$ is the total degree of $f_j(x)$ and we have adopted the usual multi-index notation $I=(i_1,\dots,i_N)$.
It follows that in $A\otimes_{\bbc}A$
$$1\otimes f_j(x) = \sum_{|I|\leq \ell_j} \left(\frac{f_j^{(I)}(x)}{I!}\otimes 1\right)\prod_{k=1}^N(1\otimes x_k-x_k\otimes 1)^{i_k},$$
and therefore in $\Gamma(U,\mathcal J^m(\mathcal E))$ we can write
$$1\otimes s = \sum_{j=1}^r\sum_{|I|\leq \min(\ell_j,m)} \left(\frac{f_j^{(I)}(x)}{I!}\otimes 1\right)\prod_{k=1}^N(1\otimes x_k-x_k\otimes 1)^{i_k}(1\otimes s_j),$$
and thus
\begin{equation}\label{decomposition of m-jet}
1\otimes s = \sum_{j=1}^r\rho_m\left(\sum_{|I|\leq \min(\ell_j,m)} \frac{f_j^{(I)}(x)}{I!}dx^I\otimes s_j\right).
\end{equation}
where here $dx^I = dx_1^{\otimes i_1}\otimes\dots\otimes dx_N^{\otimes i_N}\in \Gamma(U,\text{Sym}^m(\Omega_{X/\bbc}))$.

The dual of the bundle of $m$-jets is an $\mathcal O_X$-module whose sections are represented by linear partial differential operators acting on sections of $\mathcal E$.
\begin{defn}
The \vocab{dual of the bundle of $m$-jets} over $\mathcal E$ is the sheaf
$$\mathcal J_m(\mathcal E) = \shom_{\mathcal O_X}(\mathcal J^m(\mathcal E),\mathcal E).$$
\end{defn}
\begin{remk}
This is the dual of $\mathcal J_m(\mathcal E)$, viewed as a module over the geometric vector bundle on $X$ associated with $\mathcal E$.
\end{remk}
Note that over an affine open as described in the previous paragraph,
$$\Gamma(U,\mathcal J_m(\mathcal E)) = \vspan_{\mathcal O_X} \{\rho_m(dx^I\otimes s_j): |I|\leq m,\ 1\leq j\leq r\},$$
so an element $\chi\in \Gamma(U,\mathcal J_m(\mathcal E))$ is determined by it's values $\chi_{I,j} = \chi(\rho_m(dx^I\otimes s_j))\in\Gamma(U,\mathcal E)$.  In particular
\begin{equation}\label{linear operator formulation of dual}
\chi(j_ms) = \sum_{j=1}^r\frac{f_j^{(I)}(x)}{I!}\chi_{I,j},\quad s = \sum_{j=1}^r f_j(x)s_j
\end{equation}
which shows that $\chi$ acts as a partial differential operator locally on the sections of $\mathcal E$.

\subsection{Infinite jets}
We next define the bundle of infinite jets $\mathcal J^\infty(\mathcal E)$ of a locally free $\mathcal O_X$-module of finite rank over $X$.
In the setting of a smooth vector bundle $\pi: E\rightarrow M$ over a real manifold $M$, the infinite jet bundle $\mathcal J^\infty(\pi)$ is defined as the projective limit of the system 
of vector bundles 
$\pi_{m,\ell}: \mathcal J^m(\pi)\rightarrow\mathcal J^\ell(\pi)$ (see \cite[Chapter 7]{Saunders} for details).  The result is an infinite dimensional vector bundle 
whose fibers are Fr\'echet spaces.  Since the hom functor commutes with projective limits in the second entry, it follows that the sections of $\mathcal J^\infty(\pi)$ 
over $M$ are the projective limit of sections over $\mathcal J^m(\pi)$.
This motivates our definition of the bundle of infinite jets.
\begin{defn}
Let $\mathcal E$ be a locally free sheaf of finite rank on a scheme $X$.
The \vocab{infinite jet bundle} of $\mathcal E$ is the sheaf
$$\mathcal J^\infty(\mathcal E) = \varprojlim_m \mathcal J^m(\mathcal E)$$
where here the limit is taken over the directed system of sheaves defined by the push-forwards $\pi_{m,\ell}: \mathcal J^m(\mathcal E)\rightarrow\mathcal J^\ell(\mathcal E)$.
\end{defn}
As a sheaf, this is given by
$$\Gamma(U,\mathcal J^\infty(\mathcal E)) = \varprojlim_m \Gamma(U,\mathcal J^m(\mathcal E)) = \left\lbrace(\wt s_0,\wt s_1,\dots)\in \prod_{m=0}^\infty \Gamma(U,J^m(\mathcal E)): \pi_{m\ell}(\wt s_m) = \wt s_\ell\right\rbrace.$$
Sections of $\mathcal J^\infty(\mathcal E)$ correspond to formal Taylor series.
Alternatively, we can think of $\mathcal J^\infty(\mathcal E)$ as a pullback and push-forward involving the formal completion of $X\times X$ at the diagonal.

Each section $s$ of $\mathcal E$ over an open $U\subseteq X$ gives rise to a certain section of $\mathcal J^\infty(\mathcal E)$ in an obvious fashion.
\begin{defn}
Let $s\in \Gamma(U,\mathcal E)$ be a section.
The \vocab{$\infty$-jet of $s$} is the section of $\mathcal J^\infty(\mathcal E)$ over $U$ defined by
\begin{equation}\label{j_infty}
j^\infty(s) = (j^0(s),j^1(s),j^2(s),\dots).
\end{equation}
\end{defn}

At this point, the theory of \emph{analytic} jet bundles diverges from the theory of smooth jet bundles, since for smooth jet bundles the stalk of the infinite jet bundle is generated by $\infty$-jets of sections of $\mathcal E$.
This is not true in the analytic or algebraic situations, where the stalks may be identified with formal power series which do not necessarily converge.
We can begin to distinguish the bundle generated by infinite jets from the whole bundle via the weak bundle of infinite jets.

\begin{defn}\label{def: weak jet bundle}
Let $\mathcal E$ be a locally free sheaf of finite rank on a Cohen-Macaulay scheme $X$.
The \vocab{semi-infinite jet bundle} of $\mathcal E$ is the sheaf
$$\mathcal J^{\infty,0}(\mathcal E) = \varinjlim_m \mathcal J^m(\mathcal E).$$
where here the colimit is taken over the directed system of sheaves defined by the splitting maps $\iota_{m,\ell}: \mathcal J^\ell(\mathcal E)\rightarrow\mathcal J^m(\mathcal E)$.
\end{defn}
Over any affine open $U\subseteq X$, this presheaf is given by
$$\Gamma(U,\mathcal J^{\infty,0}(\mathcal E)) = \varinjlim_m\Gamma(U,\mathcal J^m(\mathcal E)) = \amalg_m \Gamma(U,\mathcal J^m(\mathcal E))/\sim$$
where here $\wt s_k\in \Gamma(U,\mathcal J^k(\mathcal E))$ for $k=\ell,m$ with $\ell <m$ are related by $\sim$ if and only if $\iota_{\ell,m}(s_\ell) = s_m$.
Note however that this definition does not extend globally, since colimit presheaf must be sheafified.

The semi-infinite jet bundle may be naturally identified with a subsheaf of the infinite jet bundle via
$$\mathcal J^{\infty,0}(\mathcal E)\rightarrow\mathcal J^\infty(\mathcal E),\quad s_m\mapsto (\pi_{m,0}(s_m),\dots, \pi_{m,m-1}(s_m),s_m,\iota_{m,m+1}(s_m),\dots).$$
Furthermore, over an affine ring where the algebraic functions have \emph{finite} Taylor series, $\infty$-jets of sections of $\mathcal E$ will lie in the semi-infinite jet bundle.
\begin{thm}
Let $X$ be a finite dimensional Cohen-Macaulay scheme and $\mathcal E$ a locally free $\mathcal O_X$ module on $X$ of finite rank.
Then for any affine open subset $U\subseteq X$, the infinite jet of a section $s\in\Gamma(U,\mathcal E)$ lies in the semi-infinite jet bundle $\mathcal J^{\infty,0}(\mathcal E)$.
\end{thm}
\begin{proof}
Suppose that $U=\spec(A)$ is an affine open subset of $X$ over which $\mathcal E$ is free with basis $s_1,\dots, s_r$, for $A = \mathbb C[x_1,\dots, x_N]/I$.
By \eqref{decomposition of m-jet}, 
$$1\otimes s = \rho_m\left(\sum_{j=1}^r\sum_{|I|\leq \min(\ell_j,m)} \left(\frac{f_j^{(I)}(x)}{I!}dx_1^{\otimes i_1}\otimes\dots\otimes dx_N^{\otimes i_N}\right)\otimes s_j\right).$$
Consequently $\iota_{\ell,m}(j^\ell s) = j^ms$ for all $m>\ell$ and thus $j^\infty f$ lies in $\Gamma(U,\mathcal J^{\infty,0}(\mathcal E))$.
\end{proof}



\begin{thebibliography}{xx}

\bibitem{BW}  Y. Berest and G. Wilson, {\em{Automorphisms and ideals of the Weyl algebra}}, 
Math. Ann. {\bf{318}} (2000), 127--147. 

\bibitem{BGK} M. Bergvelt, M. Gekhtman, and A. Kasman, {\em{Spin Calogero particles and bispectral solutions of the matrix KP hierarchy}},
Math. Phys. Anal. Geom. {\bf{12}} (2009), 18--200. 

\bibitem{BH} M. Bertola and J. Harnad, 
{\em{Rationally weighted Hurwitz numbers, Meijer $G$-functions and matrix integrals}}, 
J. Math. Phys. {\bf{60}} (2019), no. 10, 103504, 15 pp.

\bibitem{BCRR} A. Brini, G. Carlet, S. Romano, and P. Rossi, {\em{Rational reductions of the 2D-Toda hierarchy and mirror symmetry}}, 
J. Eur. Math. Soc. (JEMS) {\bf{19}} (2017), 835--880.

\bibitem{BC}
J. L. Burchnall and T. W. Chaundy,
{\em{Commutative ordinary differential operators}},
Proc. London Math. Soc. (2) {\bf{21}}, (1923), 420--440.

\bibitem{DG} J.~J. Duistermaat and F.~A. Gr{\"{u}}nbaum, {\em{Differential equations in the spectral parameter}}, 
Comm. Math. Phys. {\bf{103}} (1986), 177--240.

\bibitem{GW} K. R. Goodearl and R. B. Warfield, Jr.,
{\em{An introduction to noncommutative Noetherian rings}}, 2nd ed., London Math. Soc. 
Student Texts 61, Cambridge Univ. Press, 2004.

\bibitem{EGA4} A. Grothendieck and J. Dieudonn\'e, {\em{\'El\'ements de G\'eom\'etrie Alg\'ebrique IV: \'Etude locale des sch\'emas et des morphismes de sch\'emas (Quatri\'eme partie)}}, 
Inst. Hautes \`Etudes Sci. Publ. Math. 32 (1967), 5--361. Ch. IV. \S16--21.

\bibitem{Dickey}  L. A. Dickey, {\em{On the constrained KP hierarchy}}, 
Lett. Math. Phys. {\bf{34}} (1995), 379--384.

\bibitem{Hartshorne} R. Hartshorne, {\em Algebraic geometry},
Grad. Texts Math. 52, Springer-Verlag, New York-Heidelberg, 1977.

\bibitem{HvdL}  G. F. Helminck and J. W. van de Leur, {\em{An analytic description of the vector constrained KP hierarchy}}, 
Comm. Math. Phys. {\bf{193}} (1998), 627--641.

\bibitem{HMVY} C. Huang, E. Mukhin, B. Vicedo, and C. Young, 
{\em{The solutions of $\mathfrak{gl}_{M|N}$ Bethe ansatz equation and rational pseudodifferential operators}},
Selecta Math. (N.S.) {\bf{25}} (2019), no. 4, Paper No. 52, 34 pp.

\bibitem{Krichever1} I. M. Krichever, {\em{Methods of algebraic geometry in the theory of nonlinear equations}}, 
Uspehi Mat. Nauk {\bf{32}} (1977), 183--208.

\bibitem{Krichever2}  I. M. Krichever, {\em{Commutative rings of ordinary linear differential operators}}, 
Funct. Anal. Appl. {\bf{12}} (1978), 175--185.

\bibitem{Krichever3}  I. M. Krichever, {\em{General rational reductions of the Kadomtsev--Petviashvili 
hierarchy and their symmetries}},  Funct. Anal. Appl. {\bf{29}} (1995), 75--80.

\bibitem{Mumford} D. Mumford, {\em{An algebro-geometric construction of commuting operators and of
solutions to the {T}oda lattice equation, Korteweg de Vries equation and
related nonlinear equation}}, in: Proceedings of the International Symposium on Algebraic
Geometry (Kyoto Univ., Kyoto, 1977), pp. 115--153, Kinokuniya Book Store, Tokyo, 1978.

\bibitem{Sato} M. Sato and Y. Sato, {\em{Soliton equations as dynamical systems on infinite-dimensional Grassmann manifold}},
in: Nonlinear partial differential equations in applied science (Tokyo, 1982), pp. 259--271, North-Holland Math. Stud., 81, 
Lecture Notes Numer. Appl. Anal., 5, North-Holland, Amsterdam, 1983. 

\bibitem{Saunders} D. Saunders, {\em The geometry of jet bundles}, London Math. Soc. Lect. Note Ser. 142, 
Cambridge Univ. Press, 1989.

\bibitem{Schur} J. Schur, {\em{\"Uber vertauschbare lineare Differential-Ausdrücke}}, Sitzungsber. Berl. Math. Ges. {\bf{4}} (1905), 2--8.

\bibitem{Segal-Wilson} G. Segal and G. Wilson,
{\em{Loop groups and equations of {K}d{V} type}},
Inst. Hautes \'{E}tudes Sci. Publ. Math., {\bf{61}} (1985), 5--65.

\bibitem{Shiota} T. Shiota, {\em{Characterization of Jacobian varieties in terms of soliton equations}}, 
Invent. Math. {\bf{83}} (1986), 333--382.

\bibitem{Vakil} R. Vakil, {\em{A beginner's guide to jet bundles from the point of view of algebraic geometry.}},
Notes (1998).  \url{http://math.stanford.edu/~vakil/files/jets.pdf}

\bibitem{vanMoerbeke} P.~van Moerbeke, {\em{Integrable foundations of string theory}}, in: {\em
Lectures on integrable systems ({S}ophia-{A}ntipolis, 1991)}, O.~Babelon, P.~Cartier, and Y.~Kosmann-Schwarzbach, eds.,
pp. 163--267, World Sci. Publ., River Edge, NJ, 1994.

\bibitem{W1} G. Wilson, {\em{Bispectral commutative ordinary differential operators}}, J. Reine Angew. Math. {\bf{442}} (1993), 177--204.

\bibitem{W2} G. Wilson, {\em{Collisions of {C}alogero-{M}oser particles and an adelic Grassmannian}} (with an appendix by I. G. Macdonald),
Invent. Math. {\bf{133}} (1998) 1--41.
\end{thebibliography}
\end{document}